\documentclass[runningheads]{llncs}
\usepackage[margin=1.0in]{geometry}
\usepackage{graphicx,amsmath,amsfonts,bm}
\usepackage{float}
\usepackage{mathtools}
\usepackage[affil-it]{authblk}
\usepackage{amsmath,mathrsfs,amssymb}
\usepackage{mwe}
\usepackage{subfig}
\usepackage{color}

\newcommand\restr[2]{{
  \left.\kern-\nulldelimiterspace 
  #1 
  \vphantom{\big|} 
  \right|_{#2} 
  }}
\newcommand{\ou}{%
  \mathrel{%
    \vcenter{\offinterlineskip
      \ialign{##\cr$L+1$\cr\noalign{\kern-1.5pt}$1$\cr}%
    }%
  }%
}

\newcommand\restrict[3]{{
  \left.\kern-\nulldelimiterspace 
  #1 
  \vphantom{\big|} 
  \right|^{#3}_{#2} 
  }}

\newcommand{\jmp}[1]{\left[\!\left[#1\right]\!\right]}                     
\newcommand{\normH}[2]{\lVert #1\rVert_{H^{1}(#2)}}
\newcommand{\normHtwo}[2]{\lVert #1\rVert_{H^{2}(#2)}}
\newcommand{\normL}[2]{\lVert #1\rVert_{L^{2}(#2)}}
\newcommand{\normLp}[3]{\lVert #1\rVert_{L^{#3}(#2)}}
\newcommand{\normHs}[2]{\lVert #1\rVert_{H^{1}(#2)}^{2}}

\newcommand{\normLs}[2]{\lVert #1\rVert_{L^{2}(#2)}^{2}}
\newcommand{\normLinf}[2]{\lVert #1\rVert_{L^{\infty}(#2)}}

\newcommand{\normWoneinf}[2]{\lVert #1\rVert_{W^{1,\infty}(#2)}}
\newcommand{\normHj}[2]{\lVert #1\rVert_{H^{s}(#2)}}
\begin{document}

\title{Analysis of the Rigorous Coupled Wave Approach for $s$-Polarized Light in Gratings\thanks{Supported by the US National Science Foundation (NSF) under grant number DMS-1619904 and DMS-1619901.}}
\titlerunning{Analysis of the RCWA for $s$-Polarized Light in Gratings}
%

\author{Benjamin J. Civiletti\inst{1}, Akhlesh Lakhtakia
\inst{2} \and
Peter B. Monk\inst{1} }
\authorrunning{B. J. Civiletti et al.}
%

\institute{Department of Mathematical Sciences,
University of Delaware, Newark, DE \and  Department of
Engineering Science and Mechanics, Pennsylvania State University, 
University Park, PA}

\maketitle              
\begin{abstract}
We study the convergence properties of the two-dimensional Rigorous Coupled Wave Approach (RCWA) for $s$-polarized monochromatic incident light. The RCWA  is widely used to solve electromagnetic boundary-value problems where the relative permittivity varies periodically in one direction, i.e., scattering by a grating. This semi-analytical approach expands all the electromagnetic field phasors as well as the relative permittivity as Fourier series in the spatial variable along the direction of periodicity, and 
also replaces the  relative permittivity with a stairstep approximation along the direction normal to the direction of periodicity. Thus, there is error due to Fourier truncation and also due to the approximation
of grating permittivity. We prove that the RCWA  is a Galerkin scheme, which allows us to employ techniques borrowed from the Finite Element Method to analyze the error. An essential tool is a Rellich identity that shows that certain continuous problems have unique solutions that depend continuously on the data with a continuity constant having explicit dependence on the relative permittivity. We prove that the RCWA  converges with an increasing number of retained Fourier modes and with a finer approximation of the grating interfaces. Numerical results show that our convergence results for increasing the number of retained Fourier modes are seen in practice, while our estimates of convergence in slice thickness are pessimistic.  

\keywords{RCWA  \and convergence \and variational methods \and grating.}
\end{abstract}

\section{Introduction}
This paper provides an error analysis of the two-dimensional (2D) Rigorous Coupled Wave Approach (RCWA), one of several methods to solve electromagnetic scattering problems involving periodic structures \cite{Gaylord,Faryad,Akhlesh2}.  This semi-analytical approach requires all the electromagnetic field phasors as well as the relative permittivity to be expanded as Fourier series of the spatial variable along the direction of periodicity. After substitution into Maxwell's equations for time-harmonic electromagnetic fields, an infinite system of Ordinary Differential Equations (ODE) for the Fourier modes is obtained. For computational tractability, the system is truncated so that only a finite number of Fourier modes are retained. Along the direction normal to the direction of periodicity, the domain is then discretized into thin slices, and on each slice the relative permittivity is approximated by a function that is constant in the thickness direction so that the solutions to the ODEs in each slice can be obtained analytically. This allows for a fast solution algorithm to be derived \cite{Faryad,Akhlesh2,Gaylord3}, by enforcing continuity of the tangential components of electromagnetic phasors   on the inter-slice boundaries. Furthermore, suitable transmission conditions are satisfied on the top and bottom of the domain. In this way, the solutions in each slice are stitched together to form the solution on the entire domain.  

The RCWA  has its roots in coupled wave analysis for diffraction problems, e.g., in a single layer with a sinusoidal spatial variation of the relative permittivity  \cite{coupled}. The formal approach was  proposed in the early 1980s by Moharam and Gaylord \cite{Gaylord2} and a stable solution algorithm was devised several years later \cite{Gaylord3}. Subsequently, the near-field convergence with respect to the number of retained Fourier modes was drastically improved by Li \cite{Li}. The approach is now a workhorse for obtaining rapid simulations of the electromagnetic field phasors in a grating. It has been used, for example, to study the excitation of surface plasmon-polariton waves for optical sensing \cite{Homola} and in the design process of   solar cells \cite{LMAnderson,Solcore}.  Some open problems for the RCWA   were discussed by  Hench and  Strako\v{s} \cite{Hench}. One open problem discussed is whether the discretized solution approximates the true solution, and if so, to what order. We address this open problem in this paper.

The contribution of this paper is that we show that the RCWA   is a Galerkin scheme, which allows us to analyze its convergence properties. To analyze the convergence rate with respect to slice thickness, we develop an approximation theory for this type of spatial discretization.  Furthermore, we generalize a Rellich identity and an \textit{a-priori} estimate for two relevant continuous problems, and use them to show the existence and uniqueness of the solutions. To apply these continuous results to the discrete problem, we show that under certain non-trapping conditions, the continuity constant for the \textit{a-priori} estimate does not depend on slice thickness.

This paper is organized as follows. In Section \ref{Var}, we first introduce the appropriate mathematical problem: an inhomogeneous Helmholtz equation with quasi-periodic boundary conditions. After recalling the angular spectrum representation for the radiation condition, we then   give the variational formulation of our problem. In Section \ref{Rellich}, we derive a Rellich identity for the Helmholtz equation and in Section \ref{priori}, assuming the real part of the relative permittivity is positive, we give an \textit{a-priori} estimate where the continuity constant is explicit. This explicit dependence is needed both for our analysis of stairstepping, as well as in a duality argument appearing in the analysis of convergence in the number of the retained Fourier modes. This restricts us to considering non-trapping domains, as discussed later in Section \ref{priori}. The case where there is light trapping is not covered by our theory, although convergence is seen in practice \cite{Shuba1,Faiz}. In Section \ref{metallic} we show a similar \textit{a-priori} estimate holds when the real part of the relative permittivity is negative. In Section \ref{RCWA}, we show that the RCWA  is a Galerkin scheme. We then apply tools applicable to the Finite Element Method (FEM) in order to show that the RCWA  converges with respect to the number of retained Fourier modes in Section \ref{Fourier} and also with respect to the stairstep approximation  of the grating interfaces  in Section \ref{sslice}. These are the main results of the paper. Finally, in Section \ref{Numerical}, we compare the RCWA solution to a refined FEM solution to test our prediction of the order of convergence.

\section{Radiation Condition and Variational Formulation} \label{Var}
We consider linear optics with an $\exp(-i\omega t)$ dependence on time $t$, where $i=\sqrt{-1}$
and $\omega$ is the angular frequency of   light. Under this assumption, from Maxwell's equations one can show \cite{Hench,Guenther} that the electric field $\bm{E}$ solves
\begin{equation} \label{Maxwell}
\Delta \bm{E}=- \omega^{2}\mu_{0}\varepsilon_{0}\varepsilon
\bm{E}-\nabla \bigg(\bm{E} \cdot \frac{\nabla \varepsilon}{\varepsilon} \bigg), 
\end{equation}
where $\varepsilon=\varepsilon(x_{1},x_{2})$ is the spatially dependent relative permittivity,  and $\varepsilon_{0}$ and $\mu_{0}$ are the permittivity and permeability, respectively, of free space (i.e., vacuum). The domain under consideration is assumed to be invariant in the $\bm{e}_{3}=(0,0,1)$ direction, so the electric field is invariant in the $\bm{e}_{3}$ direction, i.e.

\[\bm{E}=\bm{E}(x_{1},x_{2}). \]
For $s$-polarized light, we also have that $\bm{E}=(0,0,E_{3})$, and so the last term on the right hand side of \eqref{Maxwell} is zero. The wavenumber in air is denoted by $\kappa= {\omega}/{c_{0}}$ and the speed of light in air is $c_{0}=1/\sqrt{\varepsilon_{0}\mu_{0}}$. We obtain the vector Helmholtz equation 
\[\Delta \bm{E}+ \kappa^{2}\varepsilon\bm{E}=0, \]
with $E_{1}=E_{2}=0$. So we see that this reduces to a scalar Helmholtz equation, that we study in this paper. A similar result holds for the $p$-polarization case for the magnetic field $\bm{H}$, but we do not study that problem here. To simplify the notation, from here on $E_{3}$ is denoted by $u$.

We now present the standard mathematical formulation of the basic scattering problem: a Helmholtz equation with a periodically variable relative permittivity ${\varepsilon}$. This work pertains to a 2D domain $\Omega=\{\bm{x}\in \mathbb{R}^{2}, 0<x_{1}<L_{x}, -H<x_{2}<H \},$ where $H>0$ and $L_{x}>0$. The relative permittivity $\varepsilon$ is assumed to be $L_{x}$ periodic in $x_{1}$ and invariant in $x_{3}$. 
An $s$-polarized plane wave with   electric field
phasor polarized in the $\bm{e}_{3}$ direction is incident on $\Omega$ with incidence angle $\theta$.
The third component of the incident electric field phasor can be stated as 
\[u^{\text{inc}}(x_{1},x_{2})= \exp\left[i\kappa\left(x_1\sin\theta-x_2\cos\theta\right) \right].\]
Since the structure is invariant along the $\bm{e}_{3}$ direction, the
total electric field everywhere can be stated as $u \bm{e}_{3}$, where $u$
is the solution of the Helmholtz problem
\begin{align}
    \label{eqn:PDE}
    \Delta u+\kappa^{2}{\varepsilon} u &=f \hspace{2.2cm} \text{in} \ \Omega,\\
    \exp(-i\alpha L_{x})u(0,x_{2})&=u(L_{x},x_{2}) \hspace{1cm}  \forall \ x_{2},\\
    \exp(-i\alpha L_{x})\frac{\partial }{\partial x_{2}}u(0,x_{2})&=\frac{\partial }{\partial x_{2}}u(L_{x},x_{2}) \ \ \ \forall \ x_{2},
\end{align}
where $\alpha=\kappa\sin\theta. $ Here, $f=\kappa^{2}(1-{\varepsilon})u^{i}$, but will be chosen more generally later. 

Inside $\Omega$, we assume that there are $I$ interfaces $\hat{\Gamma}_{k}$ for $1\leq k\leq I$. The interfaces are defined as
\[\hat{\Gamma}_{k}=\{\bm{x}\in\mathbb{R}^{2}, g_{k}(x_{1})=x_{2}\}, \] 
where $g_{k}:\mathbb{R}\to\mathbb{R}$ is a piecewise $C^{2}$ function except possibly at a finite number of values $x_{1k},x_{2k},\cdots,x_{N_{k}k}$. Let $\hat{H}(x)$ be the Heaviside function, and 
\[\hat{\Pi}_{ab}=\hat{H}(x_{1}-a)-\hat{H}(x_{1}-b). \] 
Then the $g_{k}$ can be written as
\begin{equation} \label{interface}
    g_{k}=\sum_{l=0}^{N_{k}}\hat{\Pi}_{x_{lk}x_{(l+1)k}}\phi_{lk}, 
\end{equation}
where the $\phi_{lk}$ are Lipschitz-continuous with $x_{0k}=0$ and $x_{(N_{k}+1)k}=L_{x}$. At the discontinuities, we require that 
\[[g_{k}]_{x_{lk}}=g_{k}(x^{+}_{lk}) -g_{k}(x^{-}_{lk}) \neq 0,\] 
for all $1\leq k \leq I$ and $1\leq l\leq N_{k}$, where $g_{k}(x^{+}_{lk})$ is the limit taken from the right and $g_{k}(x^{-}_{lk})$ is the limit taken from the left. We define the values $\mu_{lk}^{+}=\max\{ g_{k}(x^{+}_{lk}),g_{k}(x^{-}_{lk})\}$ and $\mu_{lk}^{-}=\min\{ g_{k}(x^{+}_{lk}),g_{k}(x^{-}_{lk})\}$ along
with the sets
\[W_{lk}= \{\bm{x}\in \mathbb{R}, x_{1}=x_{lk}, \mu_{lk}^{-}\leq x_{2} \leq  \mu_{lk}^{+}\}, \] 
for   $1\leq k \leq I$ and $1\leq l \leq N_{k}$. We therefore define a stairstep interface to be 
\[\Gamma_{k}=\hat{\Gamma}_{k} \cup \bigg(\bigcup_{l=1}^{N_{k}}W_{lk}\bigg). \]
An illustration of a suitable domain $\Omega$ with three interfaces is given in Fig.~1. We require that the interfaces do not intersect, so that for some $\delta>0$, we have \[\delta+\max_{0\leq x_{1}\leq L_{x}}g_{k-1}(x_{1})<g_{k}(x_{1})<-\delta+\min_{0\leq x_{1} \leq L_{x}}g_{k+1}(x_{1})\]
for all $2\leq k \leq I-1$, and the interfaces are bounded away from $\Gamma_{H}$ and $\Gamma_{-H}$, namely
\begin{align*}
    \delta-H<&g_{1}(x_{1})<-\delta+\min_{0\leq x_{1}\leq L_{x}}g_{2}(x_{1}),\\
    \max_{0\leq x_{1}\leq L_{x}}&g_{k-1}(x_{1})<g_{I}(x_{1})<-\delta+H.
\end{align*}
Thus, the interfaces $\Gamma_{k}$ separate $\Omega$ into $I+1$ subdomains, namely
\[\Omega_{k}=\{(x_{1},x_{2})\in \Omega: \ g_{k-1}(x_{1})<x_{2}<g_{k}(x_{1}) \}, \]
for $1\leq k \leq I+1$, where $g_{0}(x_{1})=-H$ and $g_{I+1}(x_{1})=H$. 

\begin{figure}[h]
	\centering
	\includegraphics[width=2.5in,height=3.5in]{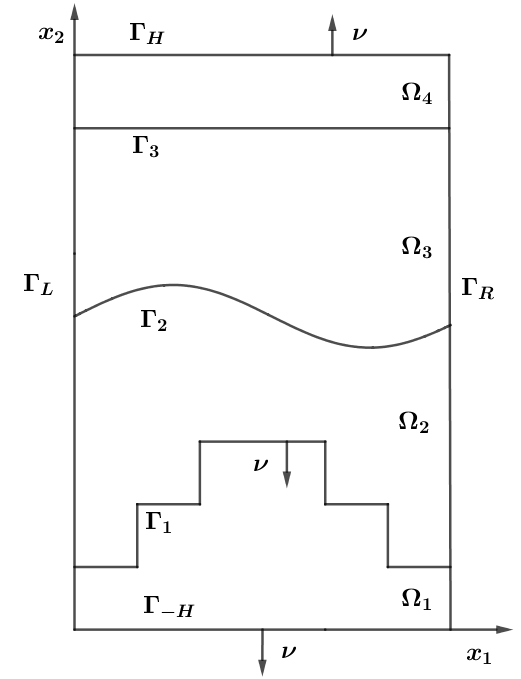}
	\caption{Geometry of the scattering problem, with $I=3$ interfaces. The domain $\Omega$ lies between the two lines $\Gamma_{H}=\{ \bm{x}\in \mathbb{R}^{2}, 0<x_{1}<L_{x}, x_{2}=H\}$ and $\Gamma_{-H}=\{ \bm{x}\in \mathbb{R}^{2}, 0<x_{1}<L_{x}, x_{2}=-H\}$. In each $\Omega_{k}$ the relative permittivity ${\varepsilon}$ is in $C^{(1,1)}$, but can jump over each interface $\Gamma_{k}$. The quasi-periodic boundaries are $\Gamma_{R}=\{\bm{x}\in \mathbb{R}^{2}, x_{1}=L_{x}, -H<x_{2}<H \}$ and $\Gamma_{L}=\{\bm{x}\in \mathbb{R}^{2}, x_{1}=0, -H<x_{2}<H \}$. The interface $\Gamma_{1}$ is a stairstep. }
\end{figure}

Furthermore, we have the following assumptions on ${\varepsilon}$. First, ${\varepsilon} \in C^{(1,1)}(\overline{\Omega_{k}})$ for all $k=1,2,\cdots, I+1$. Also, ${\varepsilon}$ is allowed to be complex valued in $\Omega$, and either $\left\{\Re ({\varepsilon}) >0, \Im ({\varepsilon}) \geq 0\right\}$ or $\left\{\Re ({\varepsilon}) \leq 0, \Im ({\varepsilon}) > 0\right\}$ in $\Omega$.  A standard assumption from the literature is that ${\varepsilon}$ is piecewise constant in $\Omega$, but we are also interested in the case where ${\varepsilon}$ is a smooth function in order to improve efficiency of solar cells \cite{Faiz,AndersonSch}. Typically, we take the relative  permittivity in the upper half space $\Omega_{H}^{+}=\{\bm{x}\in \mathbb{R}^{2}: x_{2} > H \}$ to be ${\varepsilon}^{+}=1$ and, similarly, the relative permittivity ${\varepsilon}^{-}=1$  in the lower half space $\Omega_{H}^{-}=\{\bm{x}\in \mathbb{R}^{2}: x_{2}<-H \}$. Thus, the half spaces above and below $\Omega$ are air. 

On each interface $\Gamma_{k}$, we choose the unit normal to point downwards. By $\jmp{\phi}_{\Gamma_{k}}$ we denote the jump of a function $\phi$ across the interface $\Gamma_{k}$. Thus, 
\[\jmp{\phi}_{\Gamma_{k}}= \restrict{\phi}{\Gamma_{k}}{+}-\restrict{\phi}{\Gamma_{k}}{-}\,, \]
where $\restrict{\phi}{\Gamma_{k}}{+}$ is the limit taken from $\Omega_{k+1}$ and $\restrict{\phi}{\Gamma_{k}}{-}$ is the limit taken from $\Omega_{k}$, for $1\leq k \leq I$. 

Following DeSanto \cite{DeSanto} and Chandler-Wilde \textit{et al.} \cite{Monk}, we prescribe that   $u$ can be represented  in the upper domain $\Omega_{H}^{+}$ as a linear combination of upward propagating  waves and  evanescent waves. A similar downward propagating expansion holds below the grating also, but we do not give details. For $\Gamma_{H}^{+}$ we now give a brief description of this radiation condition. Since $u$ is quasi-periodic in $\Omega$, we can write 
\begin{equation} \label{eqn:RB}
u(\bm{x})=\sum_{n\in \mathbb{Z}}u_{n}(x_{2})\exp(i\alpha_{n}x_{1}), \end{equation}
for $\bm{x} \in \Omega$, where $\alpha_{n}=\alpha+ {2\pi n}/{L_{x}}$. Now we define $\Gamma_{a}=\{\bm{x}\in \mathbb{R}^{2},x_{2}=a \},$ for $a\geq H$. More precisely, since $u$ solves the Helmholtz problem (2)--(4), we can write the Fourier coefficients of $u$ in $\Omega_{H}^{+}$ as
\begin{equation} \label{eqn:angular}
u_{n}(x_{2})=u_{n}(H)\exp \left[i(x_{2}-H)\sqrt{\kappa^{2}{\varepsilon}^{+}-\alpha_{n}^{2}}\right], \end{equation}
for all $n\in \mathbb{Z}$ and $\bm{x}\in \Omega_{H}^{+}$. From the choice that the modes need to be upward propagating or evanescent waves, we have the aforementioned angular spectrum representation for $u$,
\begin{equation} \label{eqn:Raleigh}
u(\bm{x})=\sum_{n\in \mathbb{Z}}u_{n}(H)\exp\left[i(x_{2}-H)\sqrt{\kappa^{2}{\varepsilon}^{+}-\alpha_{n}^{2}}\right]\exp(i\alpha_{n}x_{1}), 
\end{equation}
valid for all $\bm{x}\in \Omega_{H}^{+}.$ Formally taking the normal derivative of $u$ on $\Gamma_{H}$, we have

\[\restrict{\frac{\partial u}{\partial x_{2}}}{\Gamma_{H}}{}=i\sum_{n\in \mathbb{Z}}u_{n}(H)\beta_{n}\exp(i\alpha_{n}x_{1}), \]
where we assume $\alpha_{n}^{2}\neq k^{2}{\varepsilon}^{+}$ for any $n$ and

\begin{equation*}
\beta_{n}= \begin{cases} 
      \sqrt{\kappa^{2}{\varepsilon}^{+}-\alpha_{n}^{2}} & \alpha_{n}^{2}<\kappa^{2}{\varepsilon}^{+}, \\[5pt]
      i\sqrt{\alpha_{n}^{2}-\kappa^{2}{\varepsilon}^{+}} & \alpha_{n}^{2}> \kappa^{2}{\varepsilon}^{+}.
   \end{cases}
\end{equation*}
Thus, we define the Dirichlet-to-Neumann operator denoted $T^{+}$ on $\Gamma_{H}$, $T^{+}:H^{1/2}(\Gamma_{H})\to H^{-1/2}(\Gamma_{H})$, by

\[(T^{+} \phi)(x_{1})=i\sum_{n\in \mathbb{Z}}\phi_{n}\beta_{n}\exp(i\alpha_{n}x_{1}), \]
for any $\phi\in H^{1/2}(\Gamma_{H}).$
We also define the Dirichlet-to-Neumann operator $T^{-}$ in an analogous way. Now we define the space $V=S\otimes H^{1}((-H,H))$, where $S=\text{span}\{\exp(i\alpha_{n}x_{1}), -\infty<n<\infty \} \subset H^{1}((0,L_{x}))$ is the space spanned by the Fourier basis functions. We also define a truncated space 
\[ S_M=\text{span}\{\exp(i\alpha_{n}x_{1}),-M\leq n \leq M \},\] along with the space $V_{M}=S_{M}\otimes H^{1}((-H,H))$. Let $u\in V$ be a distributional solution of the Helmholtz problem (2)--(4) for a general source $f\in L^{2}(\Omega)$. Multiplying both sides of the Helmholtz equation (\ref{eqn:PDE}) by a test function $v\in V$ and integrating by parts, we get
\[ 
\int_{\Omega}\bigg( \nabla u \cdot \nabla \overline{v}-\kappa^{2}{\varepsilon} u\overline{v}\bigg) -\int_{\Gamma_{H} \cup \Gamma_{-H}}\overline{v}\nabla u \cdot \nu -\int_{\Gamma_{R}\cup \Gamma_{L}}\overline{v}\nabla u \cdot \nu=-\int_{\Omega}f\overline{v},
\]
for all $v\in V$, where the overbar denotes complex conjugation. Here, we used the fact that the integrals on the left and right boundaries cancel, because
\begin{align*}
\overline{v}_{R}\nabla u_{R}\cdot \nu_{R}&=-\big[\exp(-i\alpha L_{x})\overline{v}_{L}\big]\big[\exp(i\alpha L_{x})\nabla u_{L}\cdot \nu_{L}\big]\\
&=-\overline{v}_{L}\nabla v_{L}\cdot \nu_{L}
\end{align*} 
 follows by quasi-periodicity. The remaining normal derivatives can be replaced using the Dirichlet-to-Neumann operator. This leads to the variational problem of finding $u\in V$ such that
\begin{equation}
    b_{{\varepsilon}}(u,v)=-\int_{\Omega}f \overline{v}
    \label{eqn:problem1}
\end{equation}
for all $v\in V$, where the sesquilinear form $b_{{\varepsilon}}(\cdot,\cdot)$ is defined as
\begin{equation}
b_{{\varepsilon}}(u,v)=\int_{\Omega} \bigg( \nabla u \cdot \nabla \overline{v}-\kappa^{2}{\varepsilon} u\overline{v}\bigg) -\int_{\Gamma_{H}}\overline{v}T^{+}(u)-\int_{\Gamma_{-H}}\overline{v}T^{-}(u).
\label{eqn:sequi}
\end{equation}

Problem \eqref{eqn:problem1} uses the true relative permittivity. However we are also concerned with a second variational problem wherein ${\varepsilon}$ is replaced by an approximation ${\varepsilon}_h$. To define this approximation, the domain $\Omega$ is discretized into $S\geq 1$ slices in the $x_{2}$-direction. The slices are given by
\[S_{j}=\{\bm{x}\in \mathbb{R}^{2}, \ h_{j-1}\leq x_{2} <h_{j} \}, \]
such that $\Omega=\cup_{j=1}^{S}S_{j}.$ The thickness of each slice is $\Delta h_{j}=h_{j}-h_{j-1}$, and so we define $h=\max_{j}\Delta h_{j}$. We also require that anywhere there is $\frac{d}{dx_{1}}g_{k}=0$,
the slices are chosen so that this occurs at an inter-slice boundary. 

We assume there is a constant $C_{\Delta}>0$ such that 
\[\frac{h}{\min_{j}\Delta h_{j}}\leq C_{\Delta}.\]
 In any slice where ${\varepsilon}$ is piecewise constant and the interfaces are already a stairstep, no approximation is made. Otherwise, the true grating interface is sampled along the center line of the slice, where $x_{2}=h_{j-\frac{1}{2}}$ for $h_{j-\frac{1}{2}}=(h_{j}+h_{j-1})/2$. In each slice $S_{j}$, the true ${\varepsilon}$ is approximated as
\begin{equation}\label{eq:approx}
    {\varepsilon}_{h}(x_{1},x_{2})={\varepsilon}(x_{1},h_{j-\frac{1}{2}}).
\end{equation}
In this way, the stairstep approximation ${\varepsilon}_{h}$ is defined on $\Omega$. The key point is that ${\varepsilon}_{h}$ is independent of $x_{2}$ on each slice. A visualization of the stairstep approximation of a grating interface is given in Figure \ref{fig:approx}. 

We now define a perturbed problem with ${\varepsilon}$ replaced with ${\varepsilon}_{h}$.
We seek $u^{h}\in V$ such that 
\begin{equation}
 b_{{\varepsilon}_{h}}(u^{h},v)=-\int_{\Omega}f \overline{v}
 \label{eqn:problem2}
\end{equation}
for all $v\in V$. Here $b_{{\varepsilon}_{h}}$ is defined the same as in \eqref{eqn:sequi}, but with ${\varepsilon}$ replaced with ${\varepsilon}_{h}$. Both problems \eqref{eqn:problem1} and \eqref{eqn:problem2} have unique solutions except possibly at a discrete set of wavenumbers \cite{A-Bao}. The proof relies on compactness arguments, so the dependence of the continuity constants on ${\varepsilon}$ and ${\varepsilon}_{h}$ is unclear. In Section \ref{priori}, we derive an 
\textit{a-priori} estimate under restricted conditions, where the dependences on ${\varepsilon}$ and ${\varepsilon}_{h}$  are explicit. 

\begin{figure}[h] 
	\centering
	\includegraphics[width=4in]{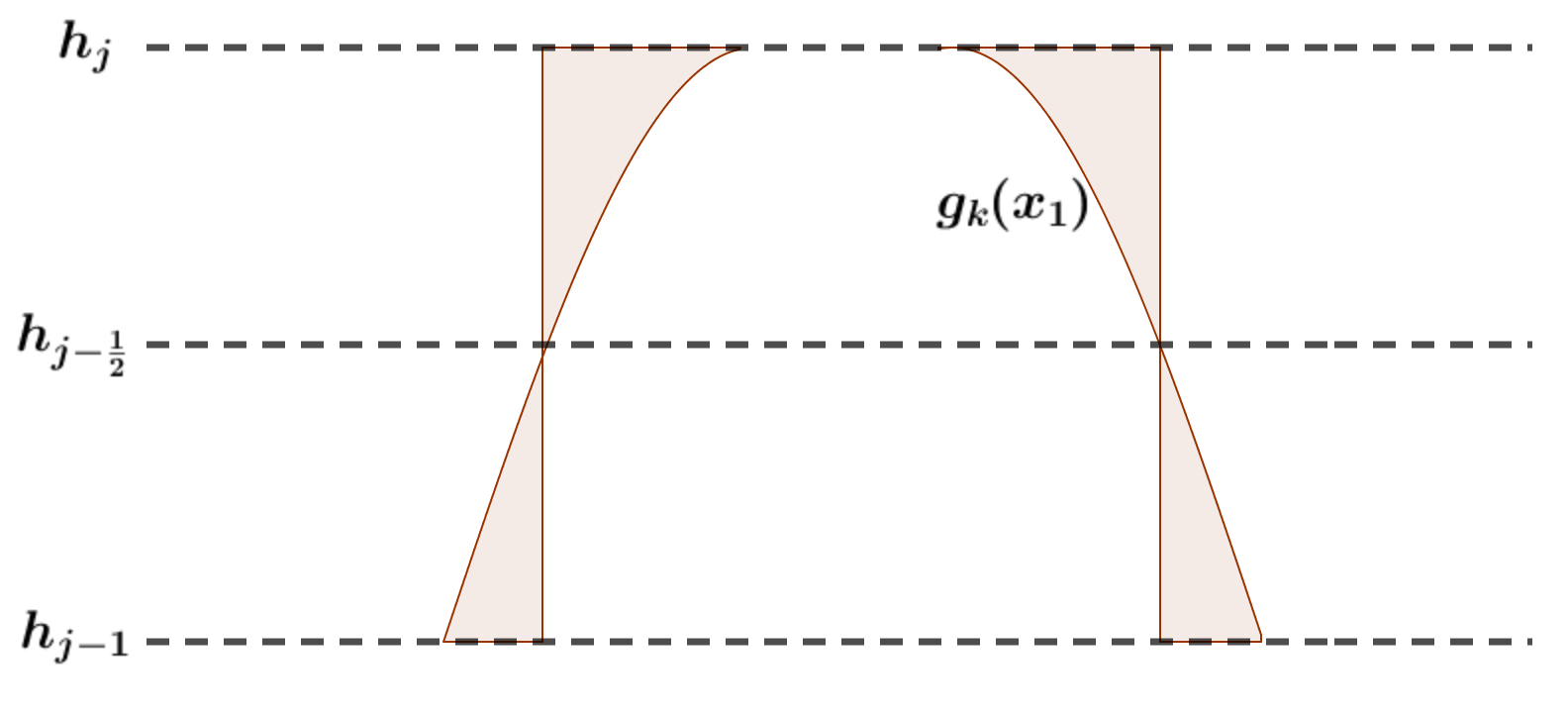}
	\caption{One slice in a grating region with interface $g_{k}(x_{1})$, showing how the stairstep approximation is made there. The shaded region is where the true relative permittivity and the approximated relative permittivity differ in the slice $S_{j}$, e.g. the $\text{supp}_{S_{j}} |{\varepsilon}-{\varepsilon}_{h}|$. Here, ${\varepsilon}$ is assumed to be piecewise constant.}
	\label{fig:approx}
\end{figure}

\section{A Rellich Identity} \label{Rellich}
The main tool to prove convergence of the RCWA   for the $s$-polarization
state is a Rellich identity for the scattering problems \eqref{eqn:problem1} and \eqref{eqn:problem2}. Later in this paper, we  use this identity to show convergence in the number of retained Fourier modes and also the slice thickness. 

We now show that a Rellich identity for an unbounded layered-media problem \cite{Lechleiter} also holds in our quasi-periodic case. Following Lechleiter and Ritterbusch \cite{Lechleiter}, we have the following lemma.

\begin{lemma}
Assume that ${\varepsilon} \in C^{(1,1)}(\overline{\Omega_{k}})$ for all $k=1,2,\cdots,I+1$ is real in $\Omega$ and $\Re({\varepsilon})> 0$. If $u$ is a solution to the variational problem \eqref{eqn:problem1} for $f\in L^{2}(\Omega)$, then the Rellich identity holds:
\begin{align*}
&\int_{\Omega}\bigg[ 2\bigg|\frac{\partial u}{\partial x_{2}}\bigg|^{2}+\kappa^{2}(x_{2}+H)\frac{\partial {\varepsilon}}{\partial x_{2}}|u|^{2}\bigg]-\sum_{k=1}^{I}\kappa^{2}\int_{\Gamma_{k}}(x_{2}+h)\jmp{{\varepsilon}}_{\Gamma_{k}}|u|^{2}\nu_{2}\\
&+2H\int_{\Gamma_{H}}\bigg(|\nabla u|^{2}-2\bigg|\frac{\partial u}{\partial x_{2}}\bigg|^{2}-\kappa^{2}{\varepsilon}|u|^{2}\bigg)
-\int_{\Gamma_{H}}\overline{u}T^{+}(u)-\int_{\Gamma_{-H}}\overline{u}T^{-}(u)\\
&=-2\int_{\Omega}(x_{2}+H)\Re(\overline{f}\frac{\partial u}{\partial x_{2}})-\int_{\Omega}f\overline{u}.
\end{align*}
\end{lemma}
\begin{remark}
Here, $\nu_{2}$ is the second component of the normal vector $\bm{\nu}$. For a stairstepped interface, the vertical sections do not appear in the sum in the first line of the Rellich identity, since $\nu_{2}=0$ there. Since the horizontal sections of a stairstep interface constitute a piecewise Lipschitz-continuous function at all but a finite number of $x_{1}$, we can control the $L^{2}$ norm of $u$.
\end{remark}

\begin{proof}
As in Ref.~\cite{Lechleiter} Lemma 3.1 (a), elliptic regularity implies that a solution $u\in H^{1}(\Omega)$ of (2)--(4) also belongs to $H^{2}(\Omega)$. Our proof follows \cite{Lechleiter}, where we check that the same Rellich identity holds for quasi-periodic solutions.
Choosing the test function $v=(x_{2}+H)\frac{\partial u}{\partial x_{2}}$, we have 
\begin{align*}
    \int_{\Omega}(x_{2}+H)\frac{\partial u}{\partial x_{2}}\Delta \overline{u}&=-\int_{\Omega}\nabla
    \bigg[(x_{2}+H)\frac{\partial u}{\partial x_{2}} \bigg] \cdot \nabla \overline{u}+\int_{\partial \Omega}(x_{2}+H)\frac{\partial u}{\partial x_{2}}\frac{\partial \overline{u}}{\partial \bm{\nu}}\nonumber\\
    &=-\int_{\Omega}\bigg|\frac{\partial u}{\partial x_{2}} \bigg|^{2}+(x_{2}+H)\nabla\bigg(\frac{\partial u}{\partial x_{2}} \bigg)\cdot \nabla \overline{u}+2H\int_{\Gamma_{H}}\bigg|\frac{\partial u}{\partial x_{2}} \bigg|^{2}. 
    \end{align*}
Here, we used Green's First Identity in the first step, and quasi-periodicity to cancel the left and right boundary integrals, since
\[ \frac{\partial u_{R}}{\partial x_{2}}\nabla\overline{u}_{R}\cdot \nu_{R}=-\frac{\partial u_{L}}{\partial x_{2}}\nabla \overline{u}_{L}\cdot \nu_{L}.\]
By taking twice the real part of both sides, and using the identity
\[\frac{\partial}{\partial x_{2}}|\nabla u|^{2}=2\Re\bigg[\nabla u \cdot \nabla \bigg(\frac{\partial \overline{u}}{\partial x_{2}} \bigg) \bigg],\]
we obtain that
\begin{align}
\label{eqn:Rellich1}
    2\Re\int_{\Omega}(x_{2}+H)\frac{\partial u}{\partial x_{2}}\Delta \overline{u}&=
    -\int_{\Omega}\bigg[2\bigg|\frac{\partial u}{\partial x_{2}} \bigg|^{2}+(x_{2}+H)\frac{\partial}{\partial x_{2}}|\nabla u|^{2}\bigg]+2H\int_{\Gamma_{H}}2\bigg|\frac{\partial u}{\partial x_{2}} \bigg|^{2}.\nonumber\\ 
    &=\int_{\Omega}\bigg(|\nabla u|^{2}-2\bigg|\frac{\partial u}{\partial x_{2}} \bigg|^{2}\bigg)+2H\int_{\Gamma_{H}}\bigg(-|\nabla u|^{2}+2\bigg|\frac{\partial u}{\partial x_{2}} \bigg|^{2}\bigg),
\end{align}
where we used the Divergence Theorem in the second step, that $x_{2}=-H$ on $\Gamma_{-H}$,
and $\nu_{2}=0$ on $\Gamma_{L}$ and $\Gamma_{R}$.

On the other hand, we have from (1) that $\Delta \overline{u}=\overline{f}-\kappa^{2}{\varepsilon} \overline{u}$ for ${\varepsilon}$ real in $\Omega$. Then,
\begin{align}    
\label{eqn:Rellich2}
    &2\Re \int_{\Omega}(x_{2}+H)\frac{\partial u}{\partial x_{2}}\Delta \overline{u}\\ \nonumber
    &=2\int_{\Omega}(x_{2}+H)\Re \bigg(\frac{\partial u}{\partial x_{2}}\overline{f} \bigg)-\kappa^{2}\int_{\Omega}(x_{2}+H){\varepsilon} 2 \Re \bigg(\frac{\partial u}{\partial x_{2}}\overline{u} \bigg)\\ \nonumber
    &=2\int_{\Omega}(x_{2}+H)\Re \bigg(\frac{\partial u}{\partial x_{2}}\overline{f} \bigg)-2H\int_{\Gamma_{H}}\kappa^{2}{\varepsilon}|u|^{2}+\kappa^{2}\int_{\Omega}\frac{\partial}{\partial x_{2}}\bigg[(x_{2}+H){\varepsilon} \bigg]|u|^{2}\\ \nonumber
    &-\sum_{k=1}^{I}\kappa^{2}\int_{\Gamma_{k}}(x_{2}+H)\jmp{{\varepsilon}}_{\Gamma_{k}}|u|^{2}\nu_{2}. \nonumber
\end{align}
This follows from integrating by parts in the second step, from the identity
\[2\Re \bigg(\frac{\partial u}{\partial x_{2}}\overline{u} \bigg)=\frac{\partial}{\partial x_{2}}|u|^{2}, \]
and by the quasi-periodicity of $u$. The Rellich identity follows from \eqref{eqn:Rellich1} and \eqref{eqn:Rellich2}.
\qed
\end{proof}

\section{An \textit{a-priori} Estimate} \label{priori}
Using the Rellich identity along the lines of \cite{Lechleiter}, we can prove an 
\textit{a-priori} estimate for the solution with a continuity constant with explicit dependence on ${\varepsilon}$ and $h$. This can be used to prove existence for all $\kappa$ under the assumptions on ${\varepsilon}$ given in the statement of the theorem in this section. The \textit{a-priori} estimate holds for all such ${\varepsilon}$ as described in Section \ref{Var}, and so it holds for the stairstep approximation ${\varepsilon}_{h}$. We rely on the non-trapping conditions to ensure that the continuity constant is bounded independent of $h$. First, we prove a lemma.

\begin{lemma} \label{ineq}
For all solutions $u\in H^{1}(\Omega)$ to the variational problem \eqref{eqn:problem1}, there is a constant $C>0$ such that
\begin{equation*}
    \normLs{u}{\Omega}\leq C\bigg(2\normLs{\frac{\partial u}{\partial x_{2}}}{\Omega}-\kappa^{2}\sum_{k=1}^{I}\int_{\Gamma_{k}}(x_{2}+H)\jmp{{\varepsilon}}_{\Gamma_{k}}|u|^{2}\nu_{2} \bigg),
\end{equation*}
where the constant 
\begin{equation} \label{constant}
C=2H\bigg(H+\frac{2}{\kappa^{2}\min_{\hat{\Gamma}_{k}}|\nu_{2}|\min_{k}\inf_{\Gamma_{k}}\big((x_{2}+H)\jmp{{\varepsilon}}_{\Gamma_{k}}\big)} \bigg). 
\end{equation}
\end{lemma}
\begin{proof}
By the definition of the $g_{k}$, we can define the subsets of $\Omega$ by
\[V_{lk}=\{\bm{x}\in\Omega, x_{lk}\leq x_{1} \leq x_{(l+1)k}, \min_{x_{lk}\leq x_{1} \leq x_{(l+1)k}} g_{k}-\delta \leq x_{2} \leq \min_{x_{lk}\leq x_{1} \leq x_{(l+1)k}} g_{k+1}-\delta \}, \]
for all $k=2,\cdots,I-1$ and all $l$. The upper bound on $x_{2}$ should be replaced with $H$ when $k=I$, and similarly the lower bound on $x_{2}$ should be $-H$ when $k=1$. By construction, we have
\[\Omega=\bigcup_{lk}V_{lk}. \]
Since each $g_{k}$ is Lipschitz-continuous in $V_{kl}$, we apply \cite{Lechleiter} Lemma 4.3 to each $V_{lk}$, so that
\[\normLs{u}{V_{lk}}\leq 4H \normLs{u}{\hat{\Gamma}_{lk}}+4H^{2}\normLs{\frac{\partial u}{\partial x_{2}}}{V_{lk}}. \]
Now we sum over all $k$ and $j$, and use that $\nu_{2}\neq 0$ on any $\hat{\Gamma}_{k}$,
\begin{align*}\normLs{u}{\Omega}&\leq \frac{4H}{\min_{\hat{\Gamma}_{k}}|\nu_{2}|}\sum_{k=1}^{I}\bigg(\normLs{|\nu_{2}|^{1/2}u}{\hat{\Gamma}_{k}}\bigg)+4H^{2}\normLs{\frac{\partial u}{\partial x_{2}}}{\Omega}\\
&\leq \frac{4H}{\kappa^{2}\min_{\hat{\Gamma}_{k}}|\nu_{2}|\min_{k}\inf_{\Gamma_{k}}\big((x_{2}+H)\jmp{{\varepsilon}}_{\Gamma_{k}} \big)}\kappa^{2}\sum_{k=1}^{I}\int_{\Gamma_{k}}(x_{2}+H)\jmp{{\varepsilon}}_{\Gamma_{k}}|u|^{2}|\nu_{2}|\\
&+4H^{2}\normLs{\frac{\partial u}{\partial x_{2}}}{\Omega},
\end{align*}
where in the last line we used that $\nu_{2}=0$ on the vertical sections of the $\Gamma_{k}$. To complete the proof, by construction we have $-\nu_{2}=|\nu_{2}|$ on $\Gamma_{k}$.
\qed
\end{proof}
Under the assumption $\Re({\varepsilon})>0$ and $\Im({\varepsilon})=0$, we prove the following theorem.
\begin{theorem} \label{apriori}
Assume that ${\varepsilon} \in C^{(1,1)}(\overline{\Omega_{k}})$ for all $k=1,2,\cdots,I+1$, and the   non-trapping conditions 
\begin{equation}\frac{\partial {\varepsilon}}{\partial x_{2}}\geq 0 \ \text{in} \ \Omega_{k}, \ \ \ \jmp{{\varepsilon}}_{\Gamma_{k}}>0, \ \text{and} \  \Re({\varepsilon}^{+}-{\varepsilon})\geq0 \ \text{on} \ \Gamma_{H}, \label{eqn:trap} \end{equation}
hold for all $k=1,2,\cdots,I+1$. Further, assume that $\Re({\varepsilon}^{\pm})>0$ and $\Im({\varepsilon}^{\pm})\geq0$ and ${\varepsilon}$ is real in $\Omega$. Then for $f\in L^{2}(\Omega)$ there exists a unique solution $u\in H^{1}(\Omega)$ of the variational problem \eqref{eqn:problem1}. Also there is an explicit constant \[C(\kappa,{\varepsilon})=C(1+\kappa^{2})\normLinf{{\varepsilon}}{\Omega}(2\rho \kappa H+4H+1)+1\] with 
$\rho=2[\Re({\varepsilon}^{+})]^{1/2}+\sqrt{2}[\Im({\varepsilon}^{+})]^{1/2}$ and $C$ defined as in \eqref{constant}, such that
\[\normH{u}{\Omega}\leq C(\kappa,{\varepsilon})\normL{f}{\Omega}.  \]
\end{theorem}
\begin{proof} The proof follows the same procedure as in \cite{Lechleiter}, but we use different Dirichlet-to-Neumann operators.
From the definition of the Dirichlet-to-Neumann operators and by Parseval's Theorem,
\[\int_{\Gamma_{H}}\overline{u}T^{+}(u)=i\sum_{n\in \mathbb{Z}}\beta_{n}|u_{n}(H)|^{2}. \]
Now we see that the signs of the real and imaginary parts of this integral are known, because 
\begin{align*}
\Re\int_{\Gamma_{H}}\overline{u}T^{+}(u)&=-\sum_{\alpha_{n}^{2}>\kappa^{2}{\varepsilon}^{+}}\sqrt{\alpha_{n}^{2}-\kappa^{2}{\varepsilon}^{+}}|u_{n}(H)|^{2},\\
\Im\int_{\Gamma_{H}}\overline{u}T^{+}(u)&=\sum_{\alpha_{n}^{2}<\kappa^{2}{\varepsilon}^{+}}\sqrt{\kappa^{2}{\varepsilon}^{+}-\alpha_{n}^{2}}|u_{n}(H)|^{2}.
\end{align*}
 For all $a\geq H$, we use the representation \eqref{eqn:angular} to compute the coefficients 
\begin{align} \label{eqn:coeff}
u_{n}(a)&=\exp\left[i(a-H)\sqrt{\kappa^{2}{\varepsilon}^{+}-\alpha_{n}^{2}}\right]
u_{n}(H),\\
(\partial_{2} u)_{n}(a)&=i\sqrt{\kappa^{2}{\varepsilon}^{+}-\alpha_{n}^{2}}\exp \left[i(a-H)\sqrt{\kappa^{2}{\varepsilon}^{+}-\alpha_{n}^{2}} \right]
u_{n}(H),\\ 
(\partial_{1} u)_{n}(a)&=i\alpha_{n}\exp \left[i(a-H)\sqrt{\kappa^{2}{\varepsilon}^{+}-\alpha_{n}^{2}}\right]u_{n}(H). \label{eqn:coeff3}
\end{align}
Furthermore, using \eqref{eqn:coeff}-\eqref{eqn:coeff3} we can bound the boundary integral on the second line of the Rellich identity,
\begin{align*} 
\int_{\Gamma_{H}}&\bigg(-|\nabla u|^{2}+2\bigg|\frac{\partial u}{\partial x_{2}}\bigg|^{2}+\kappa^{2}{\varepsilon}|u|^{2}\bigg)\\
&=\sum_{n\in \mathbb{Z}}\bigg(\big|\kappa^{2}{\varepsilon}^{+}-\alpha_{n}^{2}\big|-\alpha_{n}^{2}+\kappa^{2}{\varepsilon}^{+}\bigg)\big|
\exp \left[2i(a-H)\sqrt{\kappa^{2}{\varepsilon}^{+}-\alpha_{n}^{2}}\right]
\big||u_{}(H)|^{2}\\
&=2\sum_{\alpha_{n}^{2}<\kappa^{2}{\varepsilon}^{+}}\big(\kappa^{2}-\alpha_{n}^{2} \big)|u_{n}(H)|^{2}\\
&\leq 2k\sqrt{{\varepsilon}^{+}}\Im \int_{\Gamma_{H}}\overline{u}T^{+}(u).
\end{align*}
Now using the test function $v=u$ in the variational problem \eqref{eqn:problem1}, and taking the imaginary part, we have
\begin{align*}
    \Im\int_{\Gamma_{H}}\overline{u}T^{+}(u)&=\Im\int_{\Omega}f\overline{u}-\Im\int_{\Gamma_{-H}}\overline{u}T^{-}(u)\\
    &\leq \Im\int_{\Omega}f\overline{u}.
\end{align*}
From the non-trapping assumptions \eqref{eqn:trap} for ${\varepsilon}$ and using the estimates derived above, we get
\begin{align} \label{eqn:estimate}
&\int_{\Omega}2\bigg|\frac{\partial u}{\partial x_{2}}\bigg|^{2}-\sum_{k=1}^{I}\kappa^{2}\int_{\Gamma_{k}}(x_{2}+H)\jmp{{\varepsilon}}_{\Gamma_{k}}|u|^{2}\nu_{2}
\\\nonumber
&\leq 4kH\sqrt{{\varepsilon}^{+}}\Im \int_{\Gamma_{H}}f\overline{u}-2\int_{\Omega}(x_{2}+H)\Re(\overline{f}\partial_{2}u)-\Re \int_{\Omega}f\overline{u}.
\end{align}
Now we combine \eqref{eqn:estimate} and lemma \ref{ineq} to obtain
\begin{align*}
\normLs{u}{\Omega}&\leq C\bigg[
2\normLs{\frac{\partial u}{\partial x_{2}}}{\Omega}-\kappa^{2}\sum_{k}^{I}\int_{\Gamma_{k}}(x_{2}+H)\jmp{{\varepsilon}}_{\Gamma_{k}}\nu_{2} \bigg]\\
&\leq C \bigg[
4\kappa H\sqrt{{\varepsilon}^{+}}\Im \int_{\Gamma_{H}}f\overline{u}-2\int_{\Omega}(x_{2}+H)\Re(\overline{f}\partial_{2}u)-\Re \int_{\Omega}f\overline{u}\bigg]\\
&\leq C\big[(4 \kappa H\sqrt{{\varepsilon}^{+}}+4H+1)\normL{f}{\Omega}\normH{u}{\Omega}\big],
\end{align*}
We note that for ${\varepsilon}$ with $\Re({\varepsilon}^{\pm})> 0$ and $\Im({\varepsilon}^{\pm})\geq 0$ the term $4 \kappa \sqrt{{\varepsilon}^{+}}H$ in the above $L^{2}$ estimate can be replaced by $2\rho \kappa H$, as in \cite{Lechleiter} Lemma 4.2. Taking $u=v$ in the variational problem \eqref{eqn:problem1} and taking the real part, we have 
\[\normHs{u}{\Omega}\leq (1+\kappa^{2})\normLinf{{\varepsilon}}{\Omega}\normLs{u}{\Omega}+\normL{f}{\Omega}\normL{u}{\Omega}. \]
Consequently, for all $\kappa\geq \kappa_{0}>0$, we have a constant $C(\kappa_{0},{\varepsilon})>0$
such that $\normH{u}{\Omega}\leq C(\kappa_{0},{\varepsilon})(1+\kappa^{3})\normL{f}{\Omega}.$ Therefore we obtain existence, uniqueness and boundedness of the solution $u$ to \eqref{eqn:problem1}  and the solution $u^{h}$ to \eqref{eqn:problem2}. This follows because the 
\textit{a-priori} estimate implies an inf-sup condition for $b_{{\varepsilon}}(u,v)$ and $b_{{\varepsilon}_{h}}(u,v)$ \cite{Monk}. \qed
\end{proof}
\begin{corollary}
Assume that ${\varepsilon}$ satisfies the same assumptions as Theorem \ref{apriori}, but $\Re({\varepsilon})>0$ and $\Im({\varepsilon})>0$. Then there is a constant $C_{1}(\kappa,{\varepsilon})>0$ such that 
\[\normH{u}{\Omega}\leq C_{1}(\kappa,{\varepsilon})\normL{f}{\Omega}, \]
where the constant
\[C_{1}(\kappa,{\varepsilon})=2C(\kappa,{\varepsilon})(1+\kappa^{3})+2C(\kappa,{\varepsilon})^{2}(1+\kappa^{3})^{2}\kappa^{2}\normLinf{{\varepsilon}}{\Omega}.\]
\end{corollary}
\begin{proof}
This follows from \cite{Lechleiter} Corollary 5.1.
\qed
\end{proof}
\begin{remark}
The non-trapping conditions \eqref{eqn:trap} can be altered so that the signs of the conditions are all reversed. Under those assumptions, along with $\Re({\varepsilon}^{-}-{\varepsilon})\geq 0$ on $\Gamma_{-H}$, the same \textit{a-priori} estimate holds. 
\end{remark}

In the previous corollary we provided an \textit{a-priori}  estimate for the case where $\Re({\varepsilon})>0$ and $\Im({\varepsilon})>0$. Now we prove an \textit{a-priori}  estimate for the case where $\Re({\varepsilon})\leq 0$ and $\Im({\varepsilon})>c_{1}>0$. This case is necessary to allow, for example, metallic gratings. \begin{theorem} \label{thmmetallic}
Suppose that ${\varepsilon}$ satisfies the same conditions as in Theorem \ref{apriori}, but $\Re({\varepsilon})\leq 0$ and $\Im({\varepsilon})>c_{1}>0$. Then there is a constant $C_{2}(\kappa,{\varepsilon})>0$ such that 
\[\normH{u}{\Omega}\leq C_{2}(\kappa,{\varepsilon})\normL{f}{\Omega}. \]
\end{theorem}

\begin{proof}
Since ${\varepsilon} \in C^{(1,1)}(\overline{\Omega_{k}})$ for $k=1,2,\cdots,I+1$, it follows that $-\Re({\varepsilon}) \leq \normLinf{\Re({\varepsilon})}{\Omega}.$  Thus by taking $\mathcal{C}=\frac{\normLinf{\Re({\varepsilon})}{\Omega}+1}{\sqrt{c_{1}}}$, we define the function
\[\Phi(x_{1},x_{2})=\mathcal{C}\sqrt{\Im({\varepsilon})}+\Re({\varepsilon}), \]
and note that $\Phi>0$ by our choice of $\mathcal{C}$. Now we rewrite the Helmholz equation \eqref{eqn:PDE} as
\[\Delta u +\kappa^{2}\bigg[\Phi+i\Im({\varepsilon})\bigg] u=
f-\kappa^{2}\bigg[\Re({\varepsilon})-\Phi \bigg]u. \]
As $\hat{{\varepsilon}}=\Phi+i\Im({\varepsilon})$ satisfies $\Re(\hat{{\varepsilon}})=\Phi>0$ and $\Im(\hat{{\varepsilon}})=\Im({\varepsilon})>c_{1}>0$,   by Theorem \ref{apriori} the inequality
\begin{align*}
    \normH{u}{\Omega}&\leq C(\kappa,{\varepsilon})\normL{f-\kappa^{2}\big(\Re({\varepsilon})-\Phi \big)u}{\Omega}\\
    &\leq C(\kappa,{\varepsilon})\normL{f}{\Omega}+C(\kappa,{\varepsilon})\mathcal{C}\kappa^{2}\normL{\sqrt{\Im({\varepsilon})}u}{\Omega},
\end{align*}
  follows by our choice of $\Phi$. Like before, we take the imaginary part of the variational formulation \eqref{eqn:problem1} with $v=u$, and recall that $\Im \int_{\Gamma_{H}}\overline{u}T^{+}(u)+\Im \int_{\Gamma_{-H}}\overline{u}T^{-}(u)\geq 0$,
\begin{equation} \label{ineq2}
2C(\kappa,{\varepsilon})\mathcal{C}\kappa^{2}\normL{\sqrt{\Im({\varepsilon})}u}{\Omega}\leq \big(C(\kappa,{\varepsilon})\mathcal{C}\kappa\big)^{2}\normL{f}{\Omega}+\normH{u}{\Omega}.
 \end{equation}
Finally, we obtain a similar \textit{a-priori}  estimate found in Section \ref{priori}, namely
\[\normH{u}{\Omega}\leq \max\{2,C(\kappa,{\varepsilon})\mathcal{C}\}C(\kappa,{\varepsilon}) \big(1+\kappa^{2} \big)\normL{f}{\Omega}. \]
\qed
\end{proof}
\begin{lemma} \label{Cbound}
Suppose ${\varepsilon}$ satisfies the non-trapping conditions \eqref{eqn:trap}. Then ${\varepsilon}_{h}$ also satisfies them, and for all $h>0$, $$C(\kappa,{\varepsilon}_{h})\leq C(\kappa,{\varepsilon}).$$ 
\end{lemma}
\begin{proof}
By the definition \eqref{eq:approx}, it is clear that $\normLinf{{\varepsilon}}{\Omega}\geq\normLinf{{\varepsilon}_{h}}{\Omega}$. Now for any fixed stairstep interface $\Gamma_{k}$, the jump term in the definition of $C(\kappa,{\varepsilon}_{h})$ only appears on the horizontal sections. Since $\frac{\partial {\varepsilon}}{\partial x_{2}}\geq 0$, it follows that  
\begin{align*}
    [{\varepsilon}_{h}]_{\hat{\Gamma_{k}}}&={\varepsilon}(x_{1},h_{j+\frac{3}{2}})-{\varepsilon}(x_{1},h_{j-\frac{1}{2}})\\
    &\geq |\nu_{2}|\jmp{{\varepsilon}}_{\hat{\Gamma_{k}}}\\
    &> 0
\end{align*}
on any slice $S_{j}$.
\end{proof}
\begin{remark}
(1) Since the constants $C_{1}(\kappa,{\varepsilon})$ and $C_{2}(\kappa,{\varepsilon})$ are defined in terms of $C(\kappa,{\varepsilon})$, it also holds that $C_{1}(\kappa,{\varepsilon}_{h})\leq C_{1}(\kappa,{\varepsilon}) $ and $C_{2}(\kappa,{\varepsilon}_{h})\leq C_{2}(\kappa,{\varepsilon}) $. \\
(2) If the non-trapping conditions are not satisfied, we cannot assert that the $C(\kappa,\epsilon_{h})$ is bounded independent of $h$. Indeed, if $\Im(\epsilon)=0$, it may be that $\kappa^{2}$ is an exceptional frequency for the $\epsilon_{h}$ problem. Then $C(\kappa,\epsilon_{h})$ would not be bounded. Even if $\Im(\epsilon)>0$, it may be that $C(\kappa,\epsilon_{h})$ depends poorly on $h$. In most problems this will not be the case, so we expect RCWA to converge even for trapping domains.
\end{remark}

\section{An adjoint problem} \label{metallic}
We now study an adjoint problem, related to \eqref{eqn:problem1}. Given an $f\in L^{2}(\Omega)$, let $z_{f}\in V$ be the unique solution to the adjoint problem
\begin{equation}\label{adjoint}
\overline{b_{{\varepsilon}}(\xi, z_{f})}=-\int_{\Omega}f\overline{\xi}, 
\end{equation}
for all $\xi \in V$.
The function $z_{f}$ exists and is unique because it solves the same problem to \eqref{eqn:problem1} with $\overline{f}$ on the right hand side, and the same \textit{a-priori} estimates hold. 

To analyze the regularity of the solution, we extend ${\varepsilon}$ to the left and right by periodicity, and extend above and below by including some finite subset of the half spaces $\Omega_{H}^{+}$ and $\Omega_{H}^{-}$. We fix a $\hat{\delta}>0$ such that $\hat{\delta}<L_{x}$ and define $H^{+}=H+\hat{\delta}$ and $H^{-}=-H-\hat{\delta}$. Let $\Omega^{+}= \{\bm{x} \in \Omega_{H}^{+}, x_{2}\leq H^{+} \}$ and define $\Omega^{-}$ in a similar way. Then the domain $\Omega$ extended is given by
\[\Omega^{E}=\{\bm{x}\in \mathbb{R}^{2}, -(\zeta-1)L_{x}<x_{1}<\zeta L_{x}, -H^{-}<x_{2}<H^{+} \}, \]
where $\zeta$ is the smallest positive integer such that $2\zeta+1>\frac{2H+\hat{\delta}}{L_{x}}.$ Thus we can extend ${\varepsilon}$ to $\Omega^{E}$ by recalling that ${\varepsilon}={\varepsilon}^{+}$ in $\Omega^{+}$ and ${\varepsilon}={\varepsilon}^{-}$ in $\Omega^{-}$.

\begin{figure}[h] 
	\centering
	\includegraphics[width=4in]{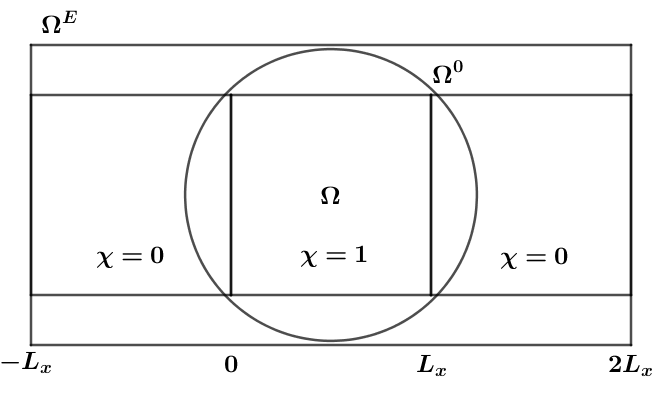}
	\caption{An illustration of the extended domain $\Omega^{E}$, with $\zeta=2$.}
	\label{fig:extended}
\end{figure}

\begin{theorem} 
\label{boundt}
Let $z_{f}$ be the solution to the adjoint problem \eqref{adjoint}. Then there exists a constant $C_{3}(\kappa)>0$ independent of ${\varepsilon}$ and $h$ such that
\[\normH{z_{f}}{\Omega^{+}} \lesssim C_{3}(\kappa)^{1/2} \normH{z_{f}}{\Omega},\]
where $C_{3}(\kappa)=H^{+}-H+(1+\kappa^{2})(1+\max_{n\in \mathbb{Z}}\frac{1}{\sqrt{\alpha_{n}^{2}-\kappa^{2}}})+\sqrt{3}(\kappa+\frac{2\pi}{L_{x}}).$
\end{theorem}
\begin{proof}
The Raleigh expansion 
\[z_{f}(\bm{x})=\sum_{n \in \mathbb{Z}}(z_{f})_{n}(H) \exp(i(x_{2}-H)\beta_{n})\exp(i\alpha_{n}x_{1}) \]
is valid for all $\bm{x}\in \Omega_{H}^{+}$. After using Parseval's Theorem, it follows that
\begin{align} \label{Raleigh}
\normHs{z_{f}}{\Omega^{+}}&=(H^{+}-H)\sum_{\alpha_{n}^{2}<\kappa^{2}}|(z_{f})_{n}(H)|^{2}(1+\kappa^{2})\\
&+\sum_{\kappa^{2}<\alpha_{n}^{2}}|(z_{f})_{n}(H)|^{2}\frac{1+\alpha_{n}^{2}}{2\Im\beta_{n}}
\left\{1-\exp\left[-2(H^{+}-H)\Im\beta_{n}\right]\right\}. \nonumber
\end{align}
The term involving $\alpha_{n}^{2}$ on the right  side of \eqref{Raleigh},
\begin{equation} \label{bound}
\frac{1+\kappa^{2}}{2\sqrt{\alpha_{n}^{2}-\kappa^{2}}}+\frac{\alpha_{n}^{2}}{2\sqrt{\alpha_{n}^{2}-\kappa^{2}}}\leq \mathscr{K}_{0}+\mathscr{K}_{1}(1+n^{2})^{1/2}
\end{equation}
for all $n$, where $\mathscr{K}_{1}=\sqrt{3}(\kappa+\frac{2\pi}{L_{x}})$, and $\mathscr{K}_{0}$ is the maximum of the first term on the left hand side of \eqref{bound}. We have used that 
\[|\alpha_{n}|<\mathscr{K}_{1}(1+n^{2})^{1/2}. \]
Then from \eqref{Raleigh}, we see that
\[\normH{z_{f}}{\Omega^{+}}\leq C_{3}(\kappa)^{1/2}
\bigg[
\sum_{n \in \mathbb{Z}}|(z_{f})_{n}(H)|^{2} (1+n^{2})^{1/2} 
\bigg]^{1/2}.\]
The proof follows by the Trace Theorem {\cite{BrennerScott}}. \qed
\end{proof}
We extend the solution $z_{f}$ to the domain $\Omega^{E}$ by quasi-periodicity to the left and right, and by the Raleigh expansion \eqref{eqn:Raleigh} above and below. We obtain the extended solution $z_{f}^{E}$ on $\Omega^{E}.$ It is useful in the in following discussion to define a restriction of $z_{f}^{E}$ to a subset of $\Omega^{E}$, namely
\[\Omega^{0}=\{ \bm{x}\in \mathbb{R}^{2}, \big|\bm{x}-(\frac{L_{x}}{2},0)\big| < H+\hat{\delta}\}, \]
such that $\Omega \subset \Omega^{0} \subset \Omega^{E}$. An illustration of this extended domain is given in Figure \ref{fig:extended}. Let $\chi$ be a smooth cutoff function such that $\chi=0$  on $\partial \Omega^{0}$ and $\chi=1$ in $\Omega$. Then the restriction $w=\chi z_{f}^{E}$ solves the Poisson problem
\begin{align} \label{Poisson}
\Delta w &= (\overline{f^{E}}-\kappa^{2}{\varepsilon}^{E}z_{f}^{E})\chi + 2 \nabla z_{f}^{E}\cdot \nabla \chi + z_{f}^{E}\Delta \chi \ \ \text{in} \ \ \Omega^{0}, \\
w&=0 \hspace{171pt} \text{on} \ \ \partial \Omega^{0}. \nonumber
\end{align}
Using this observation, we have the following corollary.
\begin{corollary} \label{cor1}
Given $f\in L^{2}(\Omega)$, then the unique solution to the adjoint problem \eqref{adjoint} $z_{f} \in H^{2}(\Omega)$.
There exists a constant $C_{4}(\kappa,{\varepsilon})>0$ with the same dependence on $\kappa$ and ${\varepsilon}$ as $C(\kappa,{\varepsilon})$, such that 
\[\normHtwo{z_{f}}{\Omega}\leq C_{4}(\kappa,{\varepsilon})\normL{f}{\Omega}. \]
\end{corollary}
\begin{proof}
    Since the extended solution $z_{f}^{E}$ solves the Poisson problem \eqref{Poisson}, by Gilbarg and Trudinger \cite{GT} there is a constant $C>0$ independent of $h$ and ${\varepsilon}$ such that
    \begin{align*}
    \normHtwo{w}{\Omega^{0}} &\leq C \normL{(\overline{f^{E}}-\kappa^{2}{\varepsilon}^{E}z_{f}^{E})\chi + 2 \nabla z_{f}^{E}\cdot \nabla \chi + z_{f}^{E}\Delta \chi}{\Omega^{0}}  \\
    &\leq (2\zeta+1) C(1+\kappa^{2}\normLinf{{\varepsilon}}{\Omega})\bigg(\normL{f}{\Omega}+\normH{z_{f}}{\Omega \cup \Omega^{+}\cup \Omega^{-}}\bigg)\\
    &\leq C_{4}(k,{\varepsilon})\normL{f}{\Omega}.
    \end{align*}
    Here we have used Theorem \ref{boundt} and the \textit{a-priori} estimate for $z_{f}$, and that the extensions are done by multiplying by phase factors. We complete the proof by recalling that \[\normHtwo{w}{\Omega^{0}}\geq\normHtwo{z_{f}}{\Omega}.\]
    \qed
\end{proof}

\section{RCWA   for $s$-polarization state} \label{RCWA}
\subsection{Description of the RCWA}
Complete descriptions of the RCWA are available elsewhere \cite{Gaylord,Faryad,Akhlesh2}. Here, we briefly describe the approach to show that the RCWA solution denoted $u^{h,M}$ solves the variational problem \eqref{eqn:problem2} with appropriate test functions. 

First, the unknown reflected and transmitted fields and the known incident field are expanded as 
Rayleigh--Bloch waves as in \eqref{eqn:RB}. For example, we expand the reflected field as
\[u^{\text{ref}}=\sum_{n\in \mathbb{Z}}u_{n}^{\text{ref}}(x_{2})\exp(i\alpha_{n}x_{1}). \]
The corresponding Fourier coefficients are given as
 \[ \{u_{n}^{\text{ref}}(x_{2})\}_{n=-\infty}^{\infty}, \{u_{n}^{\text{tr}}(x_{2})\}_{n=-\infty}^{\infty} \text{and} \ \{u_{n}^{\text{inc}}\}_{n=-\infty}^{\infty}, \]
for the reflected, transmitted, and incident fields, respectively. 
There is only one non-zero coefficient when the incident field is a plane wave, i.e., $u_{0}^{\text{inc}}=1$. The known relative permittivity ${\varepsilon}(x_{1},x_{2})$ is expanded as a Fourier series in $x_{1}$ with coefficients $\{{\varepsilon}^{n}(x_{2})\}_{n=-\infty}^{\infty}$. This representation as well as the Raleigh--Bloch expansions given by \eqref{eqn:RB} for the electromagnetic fields are substituted into Maxwell's equations. For the case where the incident plane wave is $s$-polarized, the resulting system can be written as the second-order ODE \cite{Faryad,Akhlesh2}
\begin{equation} \label{eqn:ODE}
\frac{d^{2}}{dx_{2}^{2}}u_{n}(x_{2})+\kappa^{2}\sum_{m\in \mathbb{Z}}{\varepsilon}^{(n-m)}(x_2)u_{m}(x_{2})-\alpha_{n}^{2}u_{n}(x_{2})=0,  \end{equation}
for $x_{2}\in(-H,H)$ and $n\in \mathbb{Z}$. 

To make the method computationally tractable, \eqref{eqn:ODE} needs to be truncated to retain say $2M+1$ Fourier modes. The resulting solution is denoted $u^{M}$. However, using the true ${\varepsilon}$ renders even the truncated problem difficult to solve. Thus, the RCWA introduces another discretization: a stairstep approximation of the grating interfaces using $\epsilon_{h}$. In each slice $S_{j}$, the truncated system
\begin{equation} \label{eqn:truncated}
    \frac{d^{2}}{dx_{2}^{2}}u^{h,M}_{n}(x_{2})+\kappa^{2}\sum_{m=-M}^{M}{\varepsilon}_{h}^{(n-m)} u^{h,M}_{m}(x_{2})-\alpha_{n}^{2}u^{h,M}_{n}(x_{2})=0,
\end{equation}
is solved for all $n=-M,\cdots,M$. 

As ${\varepsilon}_{h}$ is independent of $x_{2}$ in each slice, \eqref{eqn:truncated} can be solved exactly. This is used in the derivation of a fast linear algebra algorithm for computing the RCWA solution, but is not studied here. The RCWA solution $u^{h,M}(x_{1},x_{2})$ in $\Omega$ is formed by the solution in each slice along with the continuity conditions on the inter-slice boundaries
\begin{align}
    u_{n}^{h,M}(h_{j}^{-})&=u_{n}^{h,M}(h_{j}^{+}),\\
    \frac{d}{dx_{2}}u_{n}^{h,M}(h_{j}^{-})&=\frac{d}{dx_{2}}u_{n}^{h,M}(h_{j}^{+}),
\end{align}
for all $j=1,\cdots,S-1$.
Also, we have the boundary conditions on $u^{h,M}_{n}$ and its derivative 
\begin{align*}
    u^{h,M}_{n}(H^{-})&=u_{n}^{\text{inc}}+u_{n}^{\text{ref}},\\
    u_{n}^{h,M}(-H^{+})&=u_{n}^{\text{tr}},\\
    \frac{d}{dx_{2}}u^{h,M}_{n}(H^{-})&=iu_{n}^{\text{inc}}\beta_{n}+iu_{n}^{\text{ref}}\beta_{n},\\
    \frac{d}{dx_{2}}u^{h,M}_{n}(-H^{+})&=-iu_{n}^{\text{tr}}\beta_{n}.
\end{align*}
It is useful also to define the Fourier truncation operator $\mathcal{F}_{M}:V\to V_{M}$ defined as 
\[ \mathcal{F}_{M}\bigg( \sum_{n\in \mathbb{Z}}v_{n}(x_{2})\exp(i\alpha_{n}x_{1})\bigg)=\sum_{n=-M}^{M}v_{n}(x_{2})\exp(i\alpha_{n}x_{1}),\]
for all $v\in V$. We can now give a variational characterization of $u^{M}$ and $u^{h,M}$.
\begin{theorem}
The RCWA solution given by 
\[u^{h,M}(x_{1},x_{2})=\sum_{n=-M}^{M}u^{h,M}_{n}(x_{2})\exp(i\alpha_{n} x_{1}),\label{eqn:RCWA} \]
solves the variational problem 
\begin{equation} \label{eqn:problem3}
b_{{\varepsilon}_{h}}(u^{h,M},v)=-\int_{\Omega}f\overline{v},\end{equation}
for all $v\in V_{M}$.
\end{theorem}
\begin{remark}
The same result holds for a truncated solution $u^{M}(x_{1},x_{2})$ to \eqref{eqn:ODE}  with the true ${\varepsilon}$.
\end{remark}
\begin{proof}
Let $v\in V_{M}$, so that $v=\sum_{m}\xi_{m}\psi_{m}$ where $\xi_{m}\in H^{1}(-H,H)$ and $\psi_{m}\in S_{M}$ for each $-M\leq m\leq M$. We multiply both sides of \eqref{eqn:truncated} by $\overline{\xi}_{m}$, integrate by parts in $x_{2}$ on each slice, and sum over all $1\leq j \leq S$ to get 
\begin{align*}
    \int_{-H}^{H}&\bigg(\frac{d}{dx_{2}}u^{h,M}_{n}(x_{2})\frac{d}{dx_{2}}\overline{\xi}_{m}(x_{2})\bigg)-\kappa^{2}\int_{-H}^{H}\bigg(({\varepsilon}_{h} u)^{h,M}_{n}(x_{2})\overline{\xi}_{m}(x_{2})\bigg)\\
    &+\int_{-H}^{H}\bigg(\alpha_{n}^{2}u^{h,M}_{n}(x_{2})\overline{\xi}_{m}(x_{2})
    -\frac{d}{dx_{2}}u^{h,M}_{n}(H^{-})\overline{\xi}_{m}(H^{-})\\
    &+\frac{d}{dx_{2}}u^{h,M}_{n}(-H^{+})\overline{\xi}_{m}(-H^{+})\bigg)=0.
\end{align*}
Then we multiply the previous equality by $\psi_{n}\overline{\psi}_{m}$, integrate with respect to $x_{1}$, and sum over $-M\leq n \leq M $. Using the boundary conditions, we then have
\begin{align*}
    \sum_{n}\int_{0}^{L_{x}}&\bigg[\frac{d}{dx_{2}}u_{n}^{h,M}(-H^{+})\overline{\xi}_{m}(-H^{+})\psi_{n}\overline{\psi}_{m}\bigg]\\
    &=\int_{0}^{L_{x}}\bigg[\sum_{n}\frac{d}{dx_{{2}}}u_{n}^{h,M}(-H^{+})\psi_{n}\bigg]
    \bigg[ \overline{\sum_{m}\xi_{m}(-H^{+})\psi_{m}}\bigg]\\
    &=-\int_{0}^{L_{x}}\bigg[i\sum_{n}u_{n}^{\text{tr}}\beta_{n}\psi_{n}\bigg] \overline{v(-H^{+})}\\
    &=-\int_{\Gamma_{-H}}T^{-}(u^{h,M})\overline{v}.
\end{align*}
Since $\alpha^{2}\psi_{n}\psi_{m}=\big(\frac{d}{d x_{1}}\psi_{n} \big)\big(\frac{d}{d x_{1}}\overline{\psi}_{m} \big)$, it follows that
\begin{align*}
    \sum_{n}\int_{\Omega}\bigg[ &\frac{d}{dx_{2}}u_{n}^{h,M}(x_{2})\frac{d}{dx_{2}}\overline{\xi}_{m}(x_{2})\psi_{n}\overline{\psi}_{m}+\alpha_{n}^{2}u_{n}^{h,M}(x_{2})\overline{\xi}_{m}\psi_{n}\overline{\psi}_{m}\bigg]\\
    &=\int_{\Omega}\bigg\{\frac{d}{dx_{2}}
    \left[\sum_{n}u_{n}^{h,M}(x_{2})\psi_{n} \right]\frac{d}{dx_{2}}\bigg(\overline{\sum_{m}\xi_{m}\psi_{n} }\bigg)\\
    &+\frac{d}{dx_{1}}\left[\sum_{n}u_{n}^{h,M}(x_{2})\psi_{n} \right]\frac{d}{dx_{1}}\bigg(\overline{\sum_{m}\xi_{m}\psi_{n} }\bigg) \bigg\}\\
    &=\int_{\Omega}\nabla u^{h,M}\cdot \nabla \overline{v}.
\end{align*}
The other terms follow in a similar way to the two shown above.
\qed
\end{proof}

\subsection{Convergence in Number of Retained Fourier Modes} \label{Fourier}
We now prove estimates for the error due to the truncation of the Fourier series. Since $u^{h} \in H^{2}(\Omega)$, we show $O(M^{-2})$ convergence in the $L^2$ norm.
\begin{theorem} \label{modes}
Assume that ${\varepsilon}$ satisfies the non-trapping conditions \eqref{eqn:trap}. Let $u^{h}$ be a solution to the continuous problem \eqref{eqn:problem2} and $u^{h,M}$ be the RCWA solution. Then there exists a constant $C_{5}(\kappa,{\varepsilon})>0$ such that, provided $M$ is large enough,
\[\normHj{e^{h,M}}{\Omega}\leq C_{5}(\kappa,{\varepsilon}) M^{s-2}\normL{\kappa^{2}(1-{\varepsilon}_{h})u^{i}}{\Omega}, \]
for $s=0,1$. Here, $e^{h,M}=u^{h}-u^{h,M}$ is the error from Fourier truncation. 
\end{theorem}
\begin{remark}
Since we have
\begin{equation*}
\normL{\kappa^{2}(1-\varepsilon_{h})u^{i}}{\Omega} \leq \normL{\kappa^{2}(1-\varepsilon)u^{i}}{\Omega}+\normL{\kappa^{2}(\varepsilon-\varepsilon_{h})u^{i}}{\Omega},
\end{equation*}
using the upcoming lemma \ref{eperror} we see that the right hand side is bounded independent of $h$.
\end{remark}
\begin{proof}
Since ${\varepsilon}_{h}$ satisfies the non-trapping conditions \eqref{eqn:trap}, we have that $e^{h,M}\in V$ exists.
We first consider the following associated adjoint problem: for $f\in L^{2}(\Omega)$, find a $z^{h}_{f}\in V$ such that 
\[\overline{b_{{\varepsilon}_{h}}(\xi,z^{h}_{f})}=-\int_{\Omega}f \overline{\xi} \]
for all $\xi \in V$. Since $u^{h}$ solves problem \eqref{eqn:problem2} and the RCWA solution solves   problem \eqref{eqn:problem3}, we have Galerkin orthogonality in the sense that
\[b_{{\varepsilon}_{h}}(e^{h,M},z_{M})=0 \]
for all $z_{M}\in V_{M}$. Thus by taking $\xi=e^{h,M}$ in the adjoint problem and using this Galerkin orthogonality, we get
\[\overline{b_{{\varepsilon}_{h}}(e^{h,M},z^{h}_{f}-z_{M})}=-\int_{\Omega}f \overline{e^{h,M}} \] 
for all $z_{M}\in V_{M}$. Using the boundedness of the sesquilinear form $b_{{\varepsilon}_{h}}(u,v)$  and taking the infimum over all $z_{M}\in V_{M}$, we have
\begin{equation}\normL{e^{h,M}}{\Omega}\leq C_{5}\normH{e^{h,M}}{\Omega}\sup_{f\in L^{2}(\Omega)}\bigg[\frac{1}{\normL{f}{\Omega}}\inf_{z_{M}\in V_{M}}\normH{z^{h}_{f}-z_{M}}{\Omega} \bigg],\label{eqn:supinf} \end{equation}
where $C_{5}$ is the boundedness constant from $b_{{\varepsilon}_{h}}(u,v)$.
It now follows from Corollary \ref{cor1} and the standard approximation properties of Fourier series that 
\begin{align*}
  \inf_{z_{M}\in V_{M}}\normH{z^{h}_{f}-z_{M}}{\Omega}&\leq \normH{z^{h}_{f}-\mathcal{F}_{M}z^{h}_{f}}{\Omega} \\
  &\leq M^{-1}\normHtwo{z^{h}_{f}}{\Omega}\\
  &\leq C_{4}(\kappa,{\varepsilon})M^{-1}\normL{f}{\Omega}.
\end{align*}
From \eqref{eqn:supinf}, we have that 
\begin{equation} \label{Nitsche}
\normL{e^{h,M}}{\Omega}\leq C_{5} C_{4}(\kappa,{\varepsilon})M^{-1}\normH{e^{h,M}}{\Omega}. \end{equation}
We recall the sign of the real parts of the D-T-N terms, and note that for all $v\in H^{1}(\Omega),$
\[ \normHs{v}{\Omega}-(\kappa^{2}\Re({\varepsilon}_{h})+1)\normLs{v}{\Omega} \leq \Re b_{{\varepsilon}_{h}}(v,v). \]
The sesquilinear form $b_{{\varepsilon}_{h}}(u,v)$ satisfies a G\r{a}rding inequality \cite{crystals}, namely,
 \[\normHs{v}{\Omega}-C_{6}\normLs{v}{\Omega} \leq |b_{{\varepsilon}_{h}}(v,v)|,  \]
 for all $v\in H^{1}(\Omega)$, where $C_{6}=\kappa^{2}\normLinf{\Re({\varepsilon})}{\Omega}+1.$
By an argument of Schatz \cite{Schatz}, we take $v=e^{h,M}$ in the G\r{a}rding inequality, apply the Galerkin orthogonality, and divide through by $\normH{e^{h,M}}{\Omega}$ to obtain
\begin{equation} \label{Schatz}
\normH{e^{h,M}}{\Omega}-C_{6}\normL{e^{h,M}}{\Omega}\leq C_{5}\normH{u^{h}}{\Omega}.\ \end{equation}
By taking $M\geq C_{6}^{2}C_{5}C(\kappa,{\varepsilon})$ and combining \eqref{Nitsche} and \eqref{Schatz}, there is a constant $C_{7}=C_{5}C_{6}/(C_{6}-1)>0$ independent of $M,h,u^{h}$, and $u^{h,M}$ such that 
\begin{equation} \label{Schatz2}
\normH{e^{h,M}}{\Omega}\leq C_{7}\normH{u^{h}}{\Omega}. \end{equation}
Again, the standard approximation properties of Fourier series yield $\normH{e^{h,M}}{\Omega}\leq M^{-1} \normHtwo{u^{h}}{\Omega}$. It follows by \eqref{Nitsche} that 
\[\normL{e^{h,M}}{\Omega}\leq C_{5}C_{4}(\kappa,{\varepsilon}) M^{-2}\normHtwo{u^{h}}{\Omega}. \]
To complete the proof, we note that $\normHtwo{u^{h}}{\Omega}$ is bounded in terms of the data independently of $h$, due to corollary \ref{cor1} and lemma \ref{Cbound}. 
 \qed
\end{proof}

\subsection{Convergence in Slice Thickness}  \label{sslice}
This section concerns the approximation theory of the RCWA  with respect to slice thickness. For the following lemmas, we first assume that ${\varepsilon}$ is piecewise constant in each of the $\overline{\Omega}_{k}$. The case where ${\varepsilon}$ is piecewise smooth is covered later.

\begin{lemma} \label{linear}
Suppose $g_{k}$ is piecewise linear and $\varepsilon$ is piecewise constant. Then there is a constant $C_{8}>0$ independent of $h$ such that 
\begin{equation}
    \sum_{j=1}^{S} \text{meas} \ \text{supp}_{S_{j}}|{\varepsilon}-{\varepsilon}_{h}|\leq C_{8}h,
\end{equation}
for all $h>0$.
\end{lemma}
\begin{proof}
Let $\Gamma_{k}$ be a grating interface. Since $g_{k}$ is piecewise linear, for any $h>0$ the $\text{meas} \ \text{supp}_{S_{j}}|{\varepsilon}-{\varepsilon}_{h}|$ is the sum of areas of triangles. Assume that in each slice $S_{j}$, there are $2\mathscr{P}$ such triangles, where $\mathscr{P}\geq 1$ is the number of times $g_{k}$ is approximated in each $S_{j}$. Each triangle $T_{pj}$ has a horizontal side of length $t_{pj}$ such that 
\[\sum_{j=1}^{S}\sum_{p=1}^{2\mathscr{P}}t_{pj} \]
is constant for all $h>0$.
Each $t_{k,j}$ also has a vertical side with length $h/2$. Now,
\begin{align*}
\text{avg}_{pj}\text{meas} \,T_{pj}&=\frac{h}{4\mathscr{P}S} \sum_{j=1}^{S}\sum_{p=1}^{2\mathscr{P}}t_{pj}\\
&\leq\frac{L_{x}}{8H\mathscr{P}}h^{2}.
\end{align*}
Thus, the lemma follows because by definition
\begin{align*}
    \sum_{j=1}^{S} \text{meas} \ \text{supp}_{S_{j}}|{\varepsilon}-{\varepsilon}_{h}|&=S\sum_{k}\text{avg}_{pj}\text{meas} T_{k,j}\\
    &\leq\frac{I L_{x}C_{\Delta}}{4\mathscr{P}} h.
\end{align*}
\qed
\end{proof}
\begin{lemma} \label{C}
Suppose $g_{k}$ is in $C^{2}[0,L_{x}]$ and $\varepsilon$ is piecewise constant. Then there is a constant $C_{9}>0$ independent of $h$ such that 
\begin{equation}
    \sum_{j=1}^{S} \text{meas} \ \text{supp}_{S_{j}}|{\varepsilon}-{\varepsilon}_{h}|\leq C_{9}h,
\end{equation}
\end{lemma}

\begin{proof}
First, we interpolate $g_{k}$ on the inter-slice boundaries, and also on the center line as we described before. Thus, we construct a piecewise linear approximation to $g_{k}$, say $g_{k}^{*}$. Let the relative permittivity associated with $g_{k}^{*}$ in all $\Omega_{k}$ be called ${\varepsilon}^{*}$. Now we simply use the previous result, by noticing
\begin{equation*}
  \sum_{j=1}^{S} \text{meas} \ \text{supp}_{S_{j}}|{\varepsilon}-{\varepsilon}_{h}|\leq \sum_{j=1}^{S} \text{meas} \ \text{supp}_{S_{j}}|{\varepsilon}-{\varepsilon}^{*}|+C_{8}h.
\end{equation*}
Since $g_{k} \in C^{2}[0,L_{x}]$,   standard approximation theory yields 
\[|g_{k}-g_{k}^{*}|\leq \mathscr{L} h^{2} \max_{k}\max_{0\leq x_{1} \leq L_{x}}\big|\frac{d^{2}}{dx_{1}^{2}}g_{k}\big|\]
 for some $\mathscr{L}>0$ independent of h. Since the $g_{k}$ are rectifiable, 
\[\sum_{j=1}^{S} \text{meas} \ \text{supp}_{S_{j}}|{\varepsilon}-{\varepsilon}^{*}|\leq \bigg(4\mathscr{P}\mathscr{L}HC_{\Delta}I \max_{k} A(g_{k}) \max_{k}\max_{0\leq x_{1} \leq L_{x}}\big|\frac{d^{2}}{dx_{1}^{2}}g_{k}\big| \bigg)h, \]
where $A(g_{k})$ is the arclength of $g_{k}$ This inequality holds because the arclength is an upper bound on the sum of the length of $g^{*}_{k}$, and the right hand side of the inequality is the area of an approximating rectangle.
\qed
\end{proof}
\begin{lemma} \label{eperror}
Suppose ${\varepsilon} \in C^{(1,1)}(\overline{\Omega}_{k})$ for each $k$, and $g_{k}$ is piecewise $C^{2}$ on $[0,L_{x}].$ Then there is a constant $C_{10}>0$ independent of $h$ such that 
\[\normLp{\varepsilon-\varepsilon_{h}}{\Omega}{q}\leq C_{10}h^{1/q}, \]
for all $1\leq q <\infty$, and $h$ small enough.
\end{lemma}
\begin{proof}
To complete the proof of convergence in $h$, we split each slice into regions where ${\varepsilon}$ has jumps, and regions where ${\varepsilon}$ is $C^{(1,1)}$. Since we assume that any interface intersects a slice at most $\mathscr{P}$ times, this naturally separates each slice into $2\mathscr{P}+1$ regions. A visualization of a slice decomposed into the $\mathscr{P}$ regions $\mathscr{S}_{pj}$ where ${\varepsilon}$ has jumps, and the $\mathscr{P}+1$ regions $\mathscr{I}_{pj}$ where ${\varepsilon}$ is smooth is given in Figure \ref{fig:decompose}.

For each $j$, we note that we can approximate ${\varepsilon}-{\varepsilon}_{h}$ in the regions $\mathscr{I}_{pj}$ by Taylor expanding about $x_{2}=h_{j-\frac{1}{2}}$, and obtaining an $\xi_{pj}$ such that \begin{align*}
{\varepsilon}-{\varepsilon}_{h}&=\frac{\partial {\varepsilon}}{\partial x_{2}}(\xi_{pj})(x_{2}-h_{j-\frac{1}{2}})\\
&\leq \normWoneinf{{\varepsilon}}{\Omega}h.
\end{align*}
Using the previous lemmas, we can see that 
\begin{align*}
\normLp{\varepsilon-\varepsilon_{h}}{\Omega}{q}^{q}&=\sum_{j=1}^{S}\sum_{p=1}^{2\mathscr{P}+1}\bigg(\normLp{\varepsilon-\varepsilon_{h}}{\mathscr{S}_{pj}}{q}^{q}+
\normLp{\varepsilon-\varepsilon_{h}}{\mathscr{I}_{pj}}{q}^{q}\bigg)\\
&\leq (2\mathscr{P}+1)(2\normLinf{\varepsilon}{\Omega})^{q}C_{9}h+\sum_{j=1}^{S}\sum_{p=1}^{2\mathscr{P}+1}\int_{\mathscr{I}_{pj}}|\varepsilon-\varepsilon_{h}|^{q}\\
&\leq (2\mathscr{P}+1)\bigg((2\normLinf{\varepsilon}{\Omega})^{q}C_{9}+2HC_{\Delta}L_{x}\normWoneinf{{\varepsilon}}{\Omega}^{q} \bigg) (h+h^{q}).
\end{align*}
\qed
\end{proof}
\begin{figure}[h] 
	\centering
	\includegraphics[width=4in]{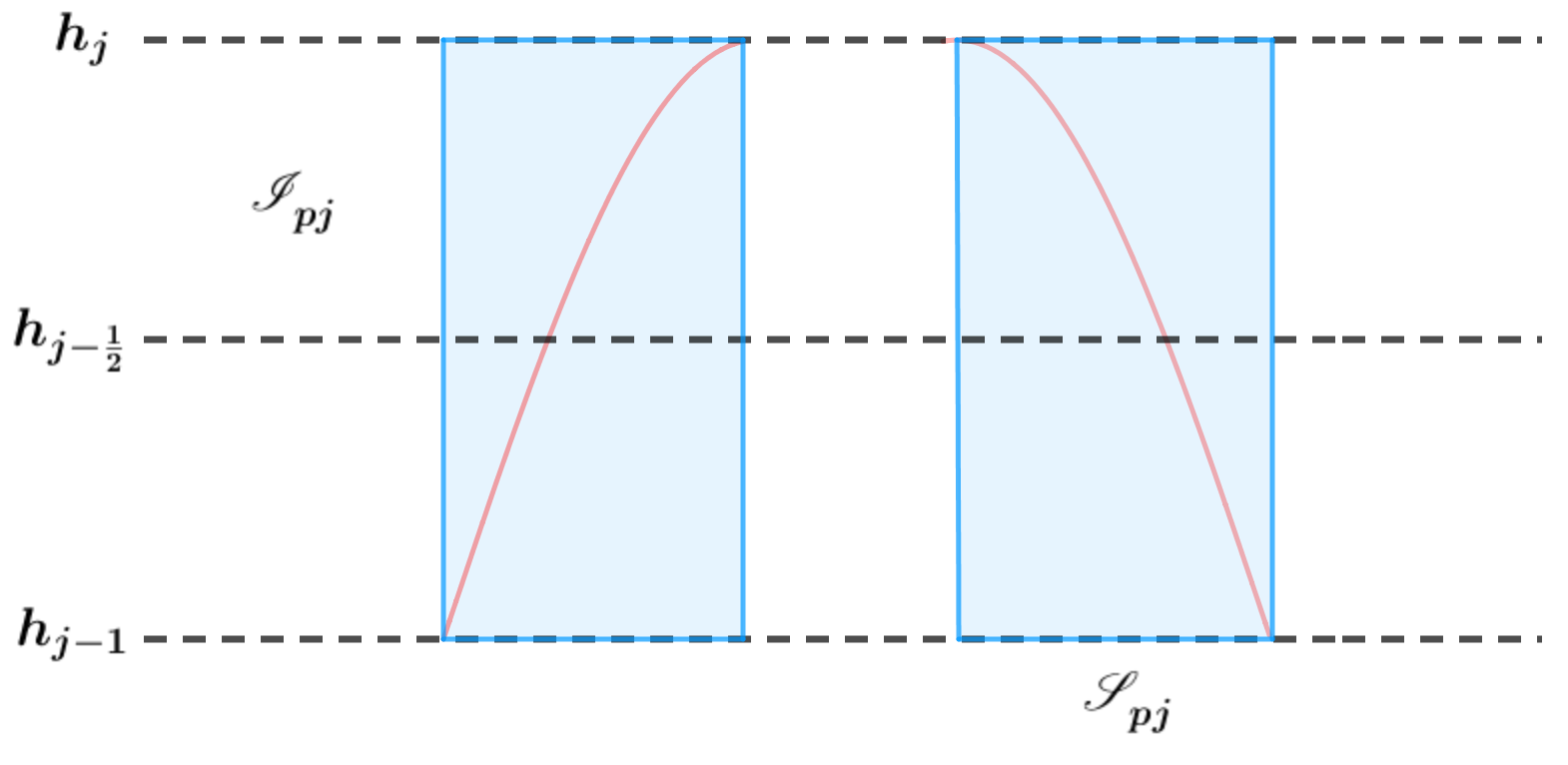}
	\caption{A single slice decomposed into five separate regions. The shaded regions are where ${\varepsilon}$ has jumps in the slice, and the white regions are where ${\varepsilon}$ is smooth.}
	\label{fig:decompose}
\end{figure}
\begin{remark}
If $\varepsilon \in C^{1,1}(\overline{\Omega})$, then for some constant $C>0$ independent of $h$, it holds that $\normLp{\varepsilon-\varepsilon_{h}}{\Omega}{q} \leq C h.$
\end{remark}
\begin{theorem} \label{slice}
Assume that ${\varepsilon}$ satisfies the non-trapping conditions \eqref{eqn:trap}. Let $u^{h}$ be a solution to the variational problem \eqref{eqn:problem2} with the stairstep approximation ${\varepsilon}_{h}$ and $u$ be the solution to \eqref{eqn:problem1} with the true ${\varepsilon}$. Assume also that the grating satisfies the conditions of any of the previous lemmas. Then there exists an explicit constant $C_{12}>0$ independent of $h$ such that
\[
\normH{e^{h}}{\Omega}\leq C_{12} h^{1/2}\normL{f}{\Omega},
\]
where $e^{h}=u-u^{h}$.
\end{theorem}
\begin{proof}
It follows from the \textit{a-priori} estimate that the two solutions $u$ and $u^{h}$ exist and are unique, and $e^{h}\in H^{1}(\Omega)$. Since $u$ and $u^{h}$ solve \eqref{eqn:PDE} with ${\varepsilon}$ and ${\varepsilon}_{h}$ respectively, we subtract the two equations to obtain
\[(\Delta+\kappa^{2}{\varepsilon}_{h})e^{h}=\kappa^{2}u({\varepsilon}_{h}-{\varepsilon}). \]
By the \textit{a-priori}  estimate there is an explicit constant $C(\kappa,{\varepsilon}_{h})>0$ depending on $\kappa$ and ${\varepsilon}_{h}$ such that
\begin{align} 
\label{eqn:happrox2}
\normH{e_{h}}{\Omega}&\leq \kappa^{2}C(\kappa,{\varepsilon}_{h}) \normLinf{u}{\Omega} \normL{{\varepsilon}_{h}-{\varepsilon}}{\Omega}\\ \nonumber
&\leq \kappa^{2}C(\kappa,{\varepsilon})\normLinf{u}{\Omega}C_{11}h^{1/2},\\ \nonumber
\end{align}
where we have used the previous lemma.We recall that $\normLinf{u}{\Omega} \leq C\normHtwo{u}{\Omega}\leq C(k,{\varepsilon}) \normL{f}{\Omega}$, by the Sobolev embedding theorem and the \textit{a-priori} estimate. \qed
\end{proof}
Combining Theorem \ref{modes} and \ref{slice} we have the following corollary.
\begin{corollary}
Under the conditions of Theorem \ref{modes} and \ref{slice}, there is a constant $C_{12}>0$ such that
\[\normHj{u-u^{h,M}}{\Omega}\leq C_{12}(h^{1/2}+M^{s-2}) \]
for $s=0,1$.
\end{corollary}
\section{Numerical Examples} \label{Numerical}
In this section we test Theorems \ref{modes} and \ref{slice} numerically by comparing the RCWA solution to a highly refined FEM solution. In order to avoid possible convergence enhancements due to symmetry, we study a non-symmetric grating profile. The example is shown in Figure \ref{illustration}a. We also show results for a symmetric grating, but the grating is taller to determine if the grating height effects the convergence with respect to the slice thickness $h$. In both of our examples, the relative permittivity of the fictitious metallic material is given as ${\varepsilon}_{m}=-15+4i$, while the relative permittivity of air is ${\varepsilon}_{a}=1$. The thickness of the air layer is $1500$ nm and the period $L_{x}=500$~nm
along the $x_1$ direction. In the first example, the non-symmetric grating of maximum height $50$~nm
is backed by a $50$-nm-thick metallic layer beneath it. The symmetric  grating has a maximum height of $100$ nm. A plane wave in both examples  is normally incident (i.e., $\theta=0$) and the free-space wavelength $\lambda_{0}=2\pi/\kappa=600$~nm. 

\begin{figure}%
    \centering
    \subfloat[]{{\includegraphics[width=5cm]{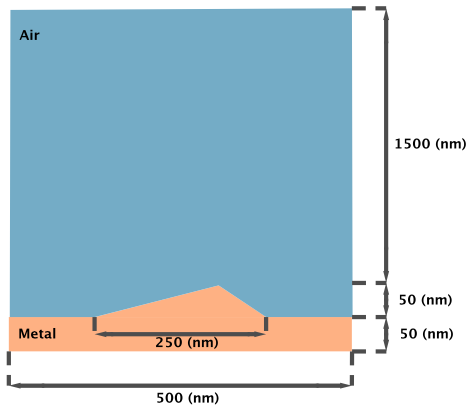} }}%
    \qquad
    \subfloat[]{{\includegraphics[width=6cm]{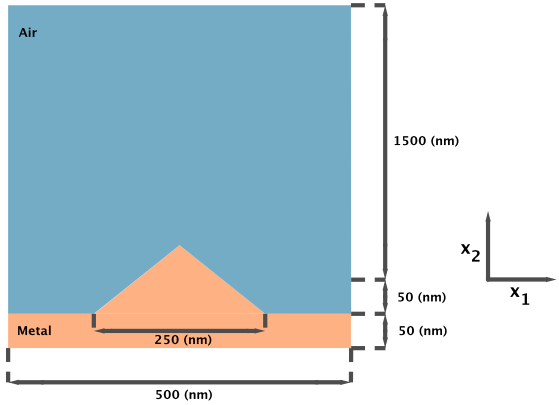} }}%
    \caption{(a) Non-symmetric grating of maximum height 50~nm.
 The peak of the grating is off center to the right by $62.5$ nm. (b) Symmetric grating
 of maximum height  $100$~nm.}%
    \label{illustration}%
\end{figure}

Since the true solution to these problems cannot be computed analytically, we compare the RCWA solution to a highly refined FEM solution. The FEM solution $u_{\text{FE}}$ in each example was computed using an adaptive method implemented in NGSolve \cite{NGSolve}. The simulated domain is sandwiched between two perfectly matched layers (PMLs). Both of the PMLs are one wavelength thick and have a constant PML parameter of $1.5 + 2.5i$ \cite{Chen}. This gives a reflection coefficient of $3 \times 10^{-12}$. The FEM solution was computed using 5th-order continuous finite elements. The adaptive algorithm uses mesh bisection and the Zienkiewicz--Zhu \textit{a-posteriori} error estimator \cite{ZZ}. Mesh adaptivity terminates when the algorithm reaches 100,000 degrees of freedom. We define the relative $L^{2}$ error between an RCWA solution and the FEM solution to be 
$$\frac{\normL{u^{h,M}-u_{\text{FE}}}{\Omega}}{\normL{u_{\text{FE}}}{\Omega}}. $$

Figures \ref{fig:fig_a} and \ref{fig:fig_b} show the convergence of the non-symmetric example with respect to $M$ and $h$, respectively. Figures \ref{fig:fig_c} and \ref{fig:fig_d} show the convergence of the symmetric example, similarly in Figs. \ref{fig:fig_b}  and  \ref{fig:fig_d}, the number of retained Fourier modes was fixed as $2M+1 =101$. Slice thickness $h$ was allowed to change, where $h\in \{1/2,1,1.25,2,5,10,25,50 \}$~nm. In Figs. \ref{fig:fig_a}  and  \ref{fig:fig_c}, the slice thickness $h=1$~nm was fixed but the number $2M+1$ of retained Fourier modes was allowed to change with $M=1,2,\cdots,50$.

We see that the rate of convergence is $O(h^{1.7})$ for the symmetric grating, and $O(h^{1.56})$ for the non-symmetric grating. In general, we can only prove at least $O(h^{1/2})$ in Theorem \ref{slice}, so in some cases the convergence due to stairstepping error is better than predicted.
The rate of convergence for the number of retained Fourier modes is   $O(M^{-2})$ for both examples.

For results in a complicated grating motivated by solar cell applications see \cite{Tom}. Convergence in $h$ was not considered, but $O(M^{-2})$ convergence is seen. 
\begin{figure}[H]
\begin{minipage}{.5\linewidth}
\centering
\subfloat[]{\label{fig:fig_a}\includegraphics[scale=.35]{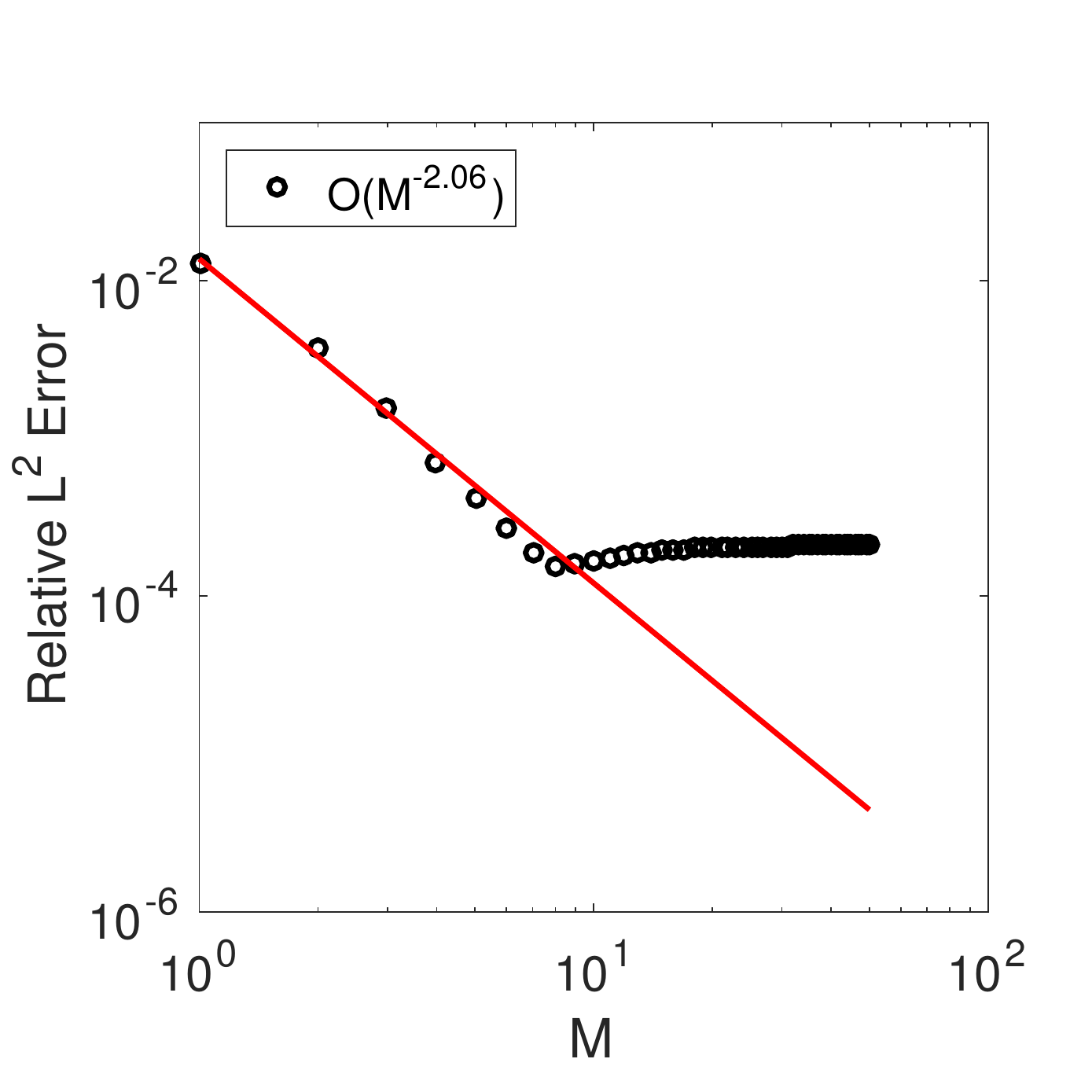}}
\end{minipage}
\begin{minipage}{.5\linewidth}
\centering
\subfloat[]{\label{fig:fig_b}\includegraphics[scale=.35]{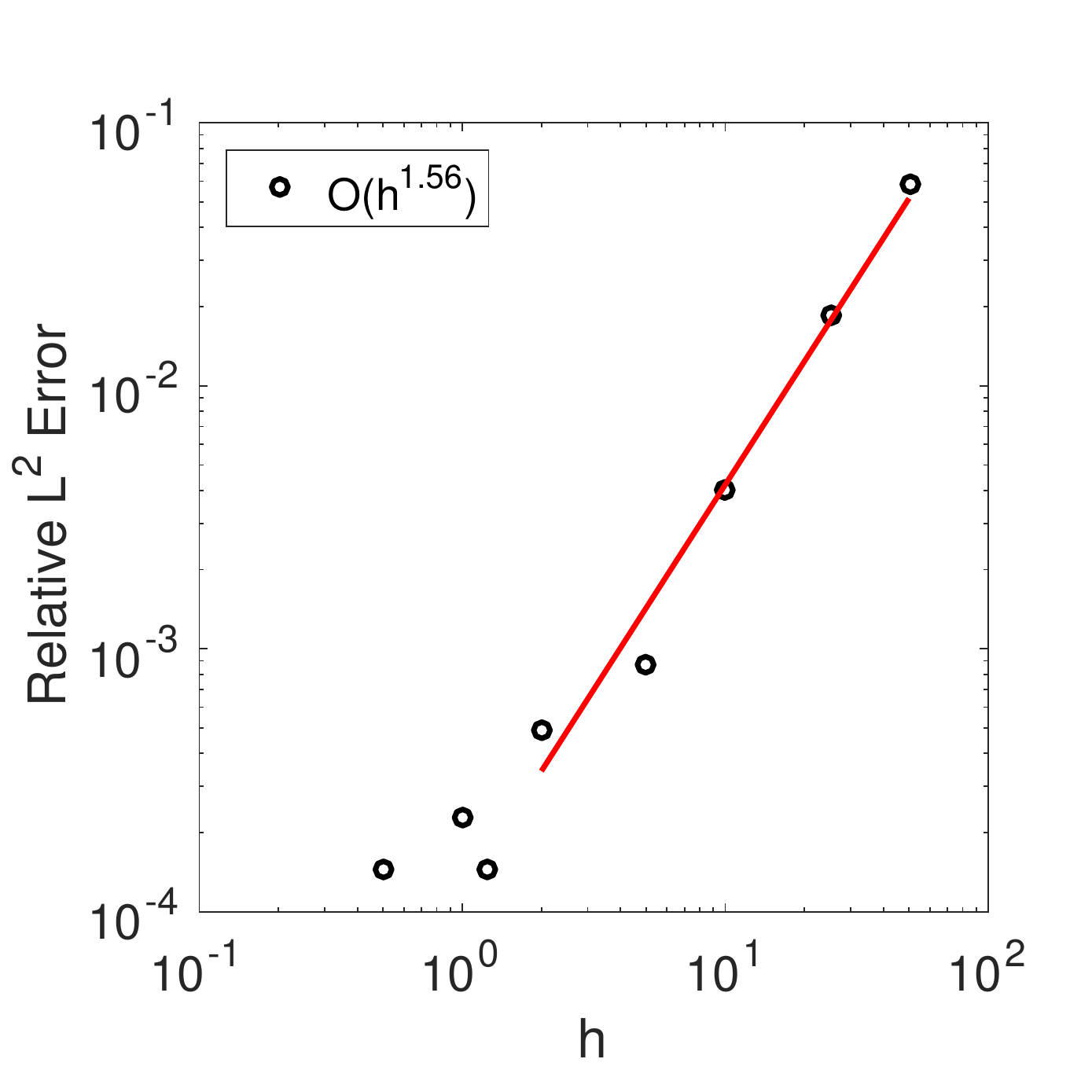}}
\end{minipage}
\begin{minipage}{.5\linewidth}
\centering
\subfloat[]{\label{fig:fig_c}\includegraphics[scale=.35]{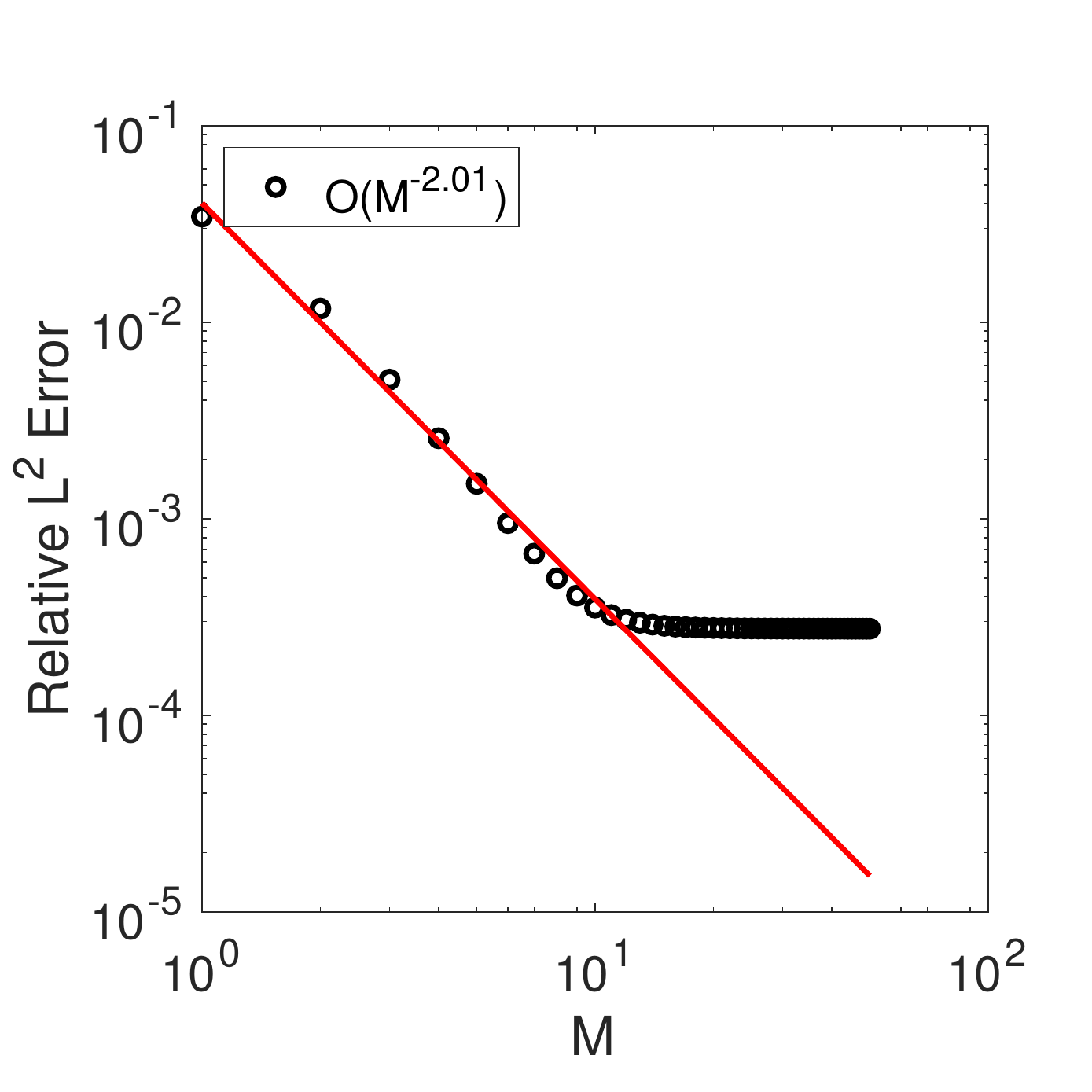}}
\end{minipage}
\begin{minipage}{.5\linewidth}
\centering
\subfloat[]{\label{fig:fig_d}\includegraphics[scale=.35]{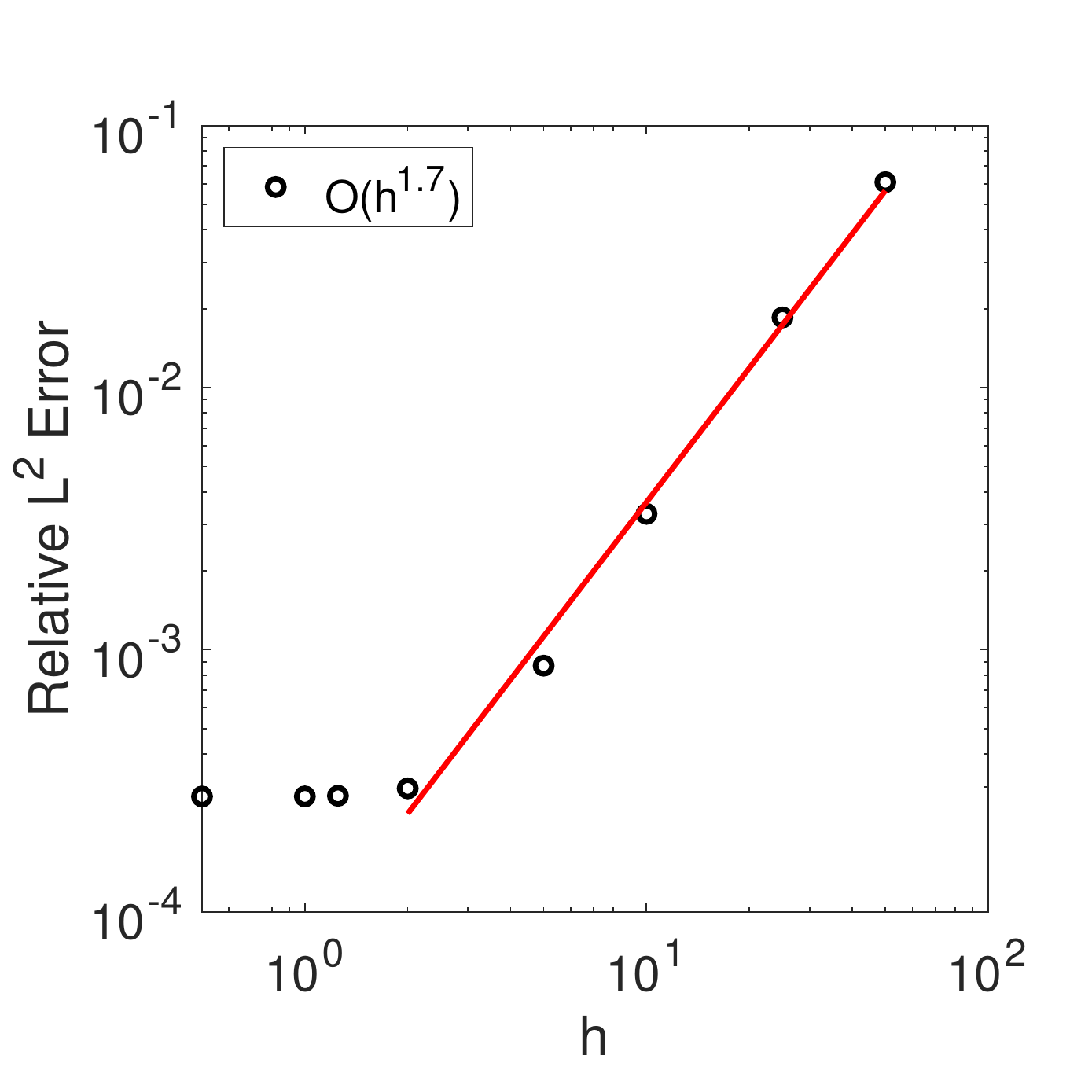}}
\end{minipage}
\caption{Convergence plots comparing the RCWA solution to a highly refined FEM solution. In (b) and (d), the number of retained Fourier modes was fixed as $2M+1 =101$. Slice thickness $h$ was allowed to change, where $h\in \{1/2,1,1.25,2,5,10,25,50 \}$~nm. In (a) and (c), the slice thickness $h=1$~nm was fixed and the number $2M+1$ of retained Fourier modes was allowed to change with $M=1,2,\cdots,50$. In all cases the error saturates around $10^{-4}$.}
\label{con}
\end{figure}

\section{Conclusion}
In this paper we studied the convergence properties of the 2D RCWA   for $s$-polarized incident light. Our analysis relies on the fact that the RCWA solution solves the appropriate variational problem, and therefore we borrowed techniques from the analysis of the FEM. Since the RCWA discretizes the solution in two different ways, we provided theorems for the convergence of the method in terms of
the number of retained Fourier modes and   slice thickness. Our analysis assumes a non-trapping domain, which is not always true for many common RCWA applications. As we commented earlier in the paper, our theory also predicts convergence in the trapping case, as long as both continuity constants in the \textit{a-priori} estimates for problems \eqref{eqn:problem1} and \eqref{eqn:problem2} are bounded independent of $h$. For problem \eqref{eqn:problem2}, the continuity constant must be bounded independent of $h$.
%
%
%

\end{document}